\documentclass[12pt]{article}
\usepackage[utf8]{inputenc}
\usepackage{graphics}
\usepackage{graphicx}
\usepackage{float}
\usepackage{tikz}
\usepackage{amsthm}
\usepackage{color}
\usepackage{xcolor}
\usepackage[a4paper, total={6.5in, 8in}]{geometry}
\usepackage{amsmath}
\usepackage{dirtytalk}
 
\usepackage{hyperref}
\hypersetup{
    colorlinks=true,
    linkcolor=blue,
    citecolor = red,
    urlcolor = false
}

\usepackage{enumerate}
\usepackage{enumitem}
\usepackage{bm}
\usepackage{bbm}
\usepackage{pstricks} 
\usepackage{pst-node}
\usepackage{pst-tree}
\usepackage{calrsfs}

\newtheorem{theorem}{Theorem}[section]
\newtheorem{remark}[theorem]{Remark}

\newtheorem{lemma}[theorem]{Lemma}

\newtheorem{definition}[theorem]{Definition}
\newtheorem{corollary}[theorem]{Corollary}

\newtheorem{proposition}[theorem]{Proposition}
\usepackage{amssymb}
\usepackage{shuffle}
\usepackage{amsfonts}
\usepackage{mathtools}

\newcommand{\Da}[1]{\mathfrak{D}(\mathfrak{a})}

\newcommand{\largetrees}{\psset{levelsep=-10pt,nodesep=-4pt,treesep=5pt}
}
\largetrees

\title{Functional representation and\\ functional calculus for controlled paths}
\author{Anna Ananova and Rama Cont}
\date{September 2026.}

\begin{document}

\maketitle

\begin{abstract}
We study the relation between non-anticipative functional calculus
and compatible families of controlled paths. For a  non-anticipative
functional  satisfying horizontal Lipschitz regularity, we show
that the iterated vertical derivatives
generate compatible higher-order controlled Taylor expansions along
H\"older controls, with level-dependent remainder exponents. We derive a rough-integration criterion
from these estimates and show that, for $\gamma-$H\"older controls with $\tfrac{1}{2}\geq \gamma>\sqrt{2}-1$, the first-order
remainder estimate is recovered.

Our main result is a converse representation theorem. We consider
non-anticipative  functionals
$G_0,\ldots,G_p$ which  satisfy
compatible higher-order controlled Taylor estimates along $\gamma$-H\"older paths. Under natural continuity and compatibility
assumptions, we prove that they can be represented as the iterated
vertical derivatives of the base functional:
\[
G_j=\nabla_\omega^jG_0,\qquad j=1,\ldots,p.
\]
Thus the Gubinelli coefficients of a compatible controlled family are symmetric and uniquely determined by its base functional; in particular, the Gubinelli derivative is identified with the vertical derivative introduced in functional It\^o calculus, giving the coefficient hierarchy an intrinsic path-space differential structure.

We show that this class of compatible coefficient families
is stable under admissible non-anticipative functional composition and derive the corresponding functional chain
rule.  As an application, we  obtain
well-posedness for a class of   path-dependent rough
differential equations with Volterra memory, and identify the
Gubinelli derivative of the resulting path-dependent rough
coefficient. 
\end{abstract} 

\newpage
\tableofcontents
\newpage
\section{Introduction}

The theory of {\it controlled paths} \cite{GUBINELLI200486,frizhairer} describes integration against irregular signals by replacing classical
smoothness of the integrand with an expansion in the increments of the underlying signal ('control').  In
Gubinelli's formulation \cite{GUBINELLI200486}, a path $Y$ controlled by a signal $X$ is equipped with coefficient
paths $Y^1,Y^2,\ldots,Y^p$ such that its increments, and those of the successive coefficients,
admit Taylor-type expansions in the increments of $X$, with regularity  graded linearly across levels: the $j$th level of
the  rough path has H\"older regularity $j\gamma$, and the corresponding remainders have the
matching linear scale prescribed by rough path theory.
The simplest example of paths controlled by $X$ is provided by smooth functions of $X$: if $Y=f(X)$ with $f\in C^p,$ then the expansions are provided by the usual Taylor expansions with coefficients $Y^k=\nabla^k f(X)$.
The functional It\^o calculus \cite{dupire2019,CF10A} and its pathwise version, the non-anticipative functional calculus \cite{CF10B,cont2012} provide Taylor-type functional expansions for path-dependent functionals \cite{ananova2017,Bielert2026,dupire2022} in terms of their {\it vertical} derivatives  and  {\it horizontal} derivatives  \cite{CF10A,dupire2019,cont2012}, and are thus natural candidates for
constructing controlled paths through functional transformations.

Several results of this type have been obtained in the literature under different assumptions.
First-order
functional expansions along a H\"older path were obtained in ~\cite[Lemma~2.2]{ananova2017}, which identified the vertical derivative as the first Gubinelli
coefficient, with a remainder of order $\gamma(1+\gamma)$ rather
than the classical order $2\gamma$.  Dupire and
Tissot-Daguette~\cite{dupire2022} have developed a different, signature-based functional Taylor
expansion involving mixed horizontal and vertical derivatives and the signature of the augmented
path.  Bielert~\cite{Bielert2026} obtained a higher-order controlled structure and change of variable formula for
H\"older controls of arbitrary regularity  under  further differentiability assumptions, also involving mixed horizontal and vertical derivatives such as
$ D_tF,D_t\nabla_\omega F,D_t\nabla_\omega^2F$ etc. More recently, Cuchiero et al. \cite{cuchiero2025} obtain expansions for functionals of c\`adl\`ag rough paths. 

The principal novelty of the present paper concerns the converse
direction. Controlled rough path theory is normally formulated
relative to a fixed control $X$. Thus, notwithstanding the
terminology of ``Gubinelli derivatives'' for the expansion
coefficients $Y^1,\ldots,Y^p$, the controlled expansion alone does
not define an intrinsic differentiation operator with respect to
$X$.   We show that,  if one considers instead \emph{families} of controlled paths  parameterized by the
underlying control $X$ then, under a  natural
continuity assumption and a compatibility condition on the expansions, the structure is rigid: the hierarchy of Gubinelli coefficients $Y^1,..,Y^p$ is then necessarily generated by
a single non-anticipative functional $G_0$, and the Gubinelli coefficients are precisely given by the
successive {\it vertical} derivatives of $G_0$ computed along $X$.  Thus the source of the rigidity is compatibility across
controls rather than a roughness or non-degeneracy condition imposed on a particular
control.

More precisely, consider non-anticipative coefficient functionals
$
G_0,G_1,\ldots,G_p$
such that, for every H\"older control $Y$, the family
\[
G_0(\cdot,Y),\,G_1(\cdot,Y),\ldots,G_p(\cdot,Y)
\]
satisfies compatible expansions in the increments of $Y$.  If the $G_\ell$ are strongly left-continuous
and the remainder in the expansion at level $\ell$ satisfies an estimate (see Theorem \ref{thm:representation}) then we have the representation
\[
G_{j}=\nabla^{j}_\omega G_0,
\qquad j=1,\ldots,p
\]
on vertical fibres over continuous stopped paths:
for every vertical perturbation $X_t^a=X_t+a 1_{[t,T]}$ of a continuous path $X$
\[
G_j(t,X_t^a)
=
D^j\bigl[b\mapsto G_0(t,X_t^b)\bigr](a),
\qquad j=1,\ldots,p.
\]
In particular, the whole coefficient hierarchy is represented by the base
functional $G_0$ which is $p$ times vertically differentiable.  Symmetry
and uniqueness of the coefficients follow immediately.  This functional representation
opens the way to applying techniques from non-anticipative functional calculus to
compatible families of controlled rough paths of arbitrary finite order.

The functional expansion results in Section \ref{sec.expansion} complement this representation theorem.  They show that
such functional controlled hierarchies arise under assumptions which are weaker in the
horizontal direction than those used in the higher-order results described above.  We do
not assume the existence of the horizontal derivatives $D_tF, D_t\nabla_\omega F,$ etc. : 
instead, the time dependence of $F$ and of its vertical derivatives is controlled only by
horizontal Lipschitz estimates.  This requires a different proof mechanism.  Rather than
applying a space-time Taylor formula to a regular approximation as in \cite{Bielert2026,dupire2022}, we work with piecewise-constant
approximations, separate horizontal and vertical increments, expand only in the vertical
directions, and optimize the approximation scale.  As we will see, the price for dispensing with horizontal differentiability is a  level-dependent family of remainder exponents.

\subsection{Overview}
Let $F$ possess $p$ vertical derivatives $
\nabla_\omega^jF, j=0,\ldots,p$ assumed to be
uniformly Lipschitz-continuous both horizontally and with respect to the supremum norm.  For a $\gamma$-H\"older path $X$, we prove, for
$\ell=0,\ldots,p-1$, a functional expansion
\[
\begin{aligned}
\nabla_\omega^\ell F(t,X_t)-\nabla_\omega^\ell F(s,X_s)
={}&
\sum_{j=1}^{p-\ell}\frac1{j!}
\nabla_\omega^{\ell+j}F(s,X_s)
[X_{s,t}^{\otimes j}]
+R^\ell_{s,t},
\end{aligned}
\]
with
\[
\|R^\ell_{s,t}\|
\lesssim
|t-s|
+
\|X\|_{\gamma;[s,t]}^{q_\ell}|t-s|^{\beta_\ell},
\]
where
\[
q_\ell
=
\frac{p-\ell+\gamma}{(p-\ell)(1-\gamma)+\gamma},
\qquad
\beta_\ell=\gamma q_\ell.
\]
The constants are uniform in the control.  Hence the vertical-derivative hierarchy has the
same algebraic form as a higher-order controlled expansion, but its remainder regularity is
not the standard linear grading across levels.  This gives a higher-order counterpart of the
 estimate in \cite{ananova2017} under purely Lipschitz horizontal regularity.  Qualitatively, the
existence of a higher-order controlled hierarchy overlaps with Bielert's result; the
hypotheses, the proof, and the resulting exponent profile are different.

For one-form-valued functionals we next determine when these estimates are sufficient for
rough integration.  If $p=\lfloor1/\gamma\rfloor$, sewing reduces to the conditions
\[
\beta_\ell+(\ell+1)\gamma>1,
\qquad \ell=0,\ldots,p-2.
\]

We then show that, in the first-order problem, the loss of
regularity produced by the approximation estimate is not
optimal. If
\[
\sqrt2-1<\gamma\le\frac12,
\qquad
X\in C^\gamma([0,T],\mathbb R^d),
\]
a second-order compensated sewing argument yields
\[
F(t,X_t)-F(s,X_s)
=
\nabla_\omega F(s,X_s)[X_{s,t}]
+
R^F_{s,t},
\qquad
\|R^F_{s,t}\|\lesssim |t-s|^{2\gamma}.
\]
Hence $
\bigl(F(\cdot,X),\nabla_\omega F(\cdot,X)\bigr)$
is a controlled path in the usual first-order sense. The
threshold $\gamma>\sqrt2-1$ arises from the sewing condition
$2\gamma+\gamma^2>1$; no other assumption on the control is required.

Finally, we prove the converse representation theorem described above  in a form which is
independent of the particular exponents arising in the forward estimates (Theorem \ref{thm:representation}).  Only a remainder
which is $o(r)$ in the size of the control increment is required.  
The resulting symmetry and uniqueness statements show that a compatible family of
Gubinelli coefficients carries an intrinsic path-space differential structure and may be represented as a functional Taylor expansion.

Our results thus establish a two-way correspondence:
functional calculus produces controlled hierarchies, while compatible
families satisfying the assumptions of the representation theorem
are generated by vertically differentiable non-anticipative
functionals.

It is well known that Gubinelli derivatives associated with a fixed control need not
be unique in general. Uniqueness may be recovered by imposing
roughness or non-degeneracy on the driving signal \cite{frizhairer,vaskovskii2022}. Our result is of a different nature. We impose no roughness condition on controls: instead, rigidity arises from compatibility
of the controlled expansions as the underlying control varies.
Under this compatibility condition, the coefficients are not
only unique: they are necessarily the  vertical
derivatives of the base functional computed along the control.

The approximation approach used below originates in our previous work \cite{ananova2017}, which provided  the first  link
between  vertical derivatives and Gubinelli-type
expansions.  There, a H\"older path is
replaced on a short interval by a bounded-variation
approximation, functional change-of-variable formulas are
applied along the approximating path, and the resulting growth
of its variation is balanced against the uniform approximation
error. The same paper also develops higher-order expansion
arguments based on repeated vertical differentiation and
symmetry of the vertical derivatives.

This approximation-and-balancing principle was subsequently
used and developed in several directions. Ananova
\cite{ananova2023} applied it to stability of controlled paths
under non-anticipative functional transformations and to
path-dependent rough differential equations, while
Bielert~\cite{Bielert2026} developed a higher-order version under stronger horizontal
differentiability assumptions. The present work follows the
same approximation philosophy but removes horizontal
derivatives from the higher-order Taylor estimates:
horizontal dependence is controlled only through Lipschitz
bounds, and the approximation is organized so as to separate
horizontal and vertical increments.
\subsection{Outline}
Section~\ref{sec:functional-calculus} defines the notation and recalls notions from
non-anticipative functional calculus and 
controlled rough paths used in the sequel.
Section~\ref{sec.expansion} studies the direct passage from functional calculus to
controlled paths.  Proposition ~\ref{prop:functional-taylor-expansion} establishes a higher-order controlled Taylor
estimate under horizontal Lipschitz regularity.  In Section \ref{sec:weak-horizontal-integration} we derive  the consequences of this expansion for
rough integration, and in Section \ref{sec:quadratic-variation-recovery} we recover
first-order controlled regularity for $\gamma$-H\"older paths with $\tfrac{1}{2}\geq \gamma> \sqrt{2}-1$.

Section~\ref{sec:converse} is devoted to the converse direction. Our main result is a
functional representation theorem for compatible families of
controlled paths (Theorem \ref{thm:representation}), which shows that their coefficients are given by the vertical derivatives of a
single non-anticipative functional.
In Section \ref{sec:rigidity} we derive symmetry and
uniqueness as consequences of this representation theorem.

Section~\ref{sec.chainrule} develops a functional calculus for such compatible
families. We derive a functional chain rule for their coefficient
hierarchies, introduce a finite-jet formulation and prove that
the abstract $o(r)$-compatible class is stable under admissible
non-anticipative functional transformations.
Combining this
stability result with the representation theorem (Thm \ref{thm:representation}) identifies the
transformed Gubinelli coefficients with the vertical derivatives
of the transformed base functional.
Section \ref{sec:application-rde} gives an application to rough differential equations with path-dependent coefficients.
Finally, Section~\ref{sec:discussion} discusses the representation
principle and its relation to other approaches to functional
calculus for controlled paths.
\section{Notations and preliminary results}
\label{sec:functional-calculus}

We collect the notions from non-anticipative functional calculus
that are used in the sequel. We refer to
\cite{ananova2017,cont2012,CF10B} for the
general theory and functional change of variable formulae.
\subsection{Stopped paths and vertical derivatives}
We use the terminology and notations of \cite{ananova2017,cont2012}.\\
$D([0,T],\mathbb R^d)$ denotes the space of
c\`adl\`ag paths and, for $X\in D([0,T],\mathbb R^d)$,
\[
X_t(r):=X(r\wedge t)
\]
denotes the path stopped at time $t$.
The space of stopped paths is the quotient space
\[
\Lambda_T^d
:=
\bigl([0,T]\times D([0,T],\mathbb R^d)\bigr)/{\sim},
\]
for the equivalence relation
\[
(t,X)\sim(t',X')
\quad\Longleftrightarrow\quad
t=t'
\ \text{ and }\ 
X_t=X'_{t'}.
\]
We write a stopped path simply as $(t,X_t)$ and equip
$\Lambda_T^d$ with the metric
\[
d_\infty\bigl((t,X_t),(s,Y_s)\bigr)
:=
|t-s|+\|X_t-Y_s\|_\infty.
\]

We denote 
$
W_T^d\subset\Lambda_T^d $
the subset of stopped paths having a continuous representative.

A map $
F:\Lambda_T^d\longrightarrow E,$
where $E$ is a finite-dimensional normed space, is called
non-anticipative. Equivalently,
\[
F(t,X)=F(t,X_t).
\]
For $e\in\mathbb R^d$ and $(t,X_t)\in\Lambda_T^d$, define
the vertical perturbation
\[
X_t^e
:=
X_t+e\,\mathbf 1_{[t,T]}.
\]
For a non-anticipative functional
$F:\Lambda_T^d\to E$ and $e\in\mathbb R^d$ we denote
\[
g_{t,X}(e):=F(t,X_t^e)=F(t,X_t+e\,\mathbf 1_{[t,T]}).
\]

\begin{definition}[Vertical derivatives]
\label{def:vertical-derivatives}
The functional $F$ is vertically differentiable at $(t,X_t)$
if $g_{t,X}$ is Fr\'echet differentiable at $0$. Its vertical
derivative \cite{dupire2019,CF10A} is
\[
\nabla_\omega F(t,X)
:=
Dg_{t,X}(0)
\in\operatorname{Lin}(\mathbb R^d,E).
\]

Iterating this construction, whenever the derivatives exist,
we set
\[
\nabla_\omega^jF(t,X)
:=
D^jg_{t,X}(0)
\in
\operatorname{Lin}
\bigl((\mathbb R^d)^{\otimes j},E\bigr),
\qquad j\ge1.
\]
\end{definition}
Since $g_{t,X}$ is a map between finite-dimensional spaces,
whenever its $j$-th Fr\'echet derivative exists,
\[
\nabla_\omega^jF(t,X)
\in
\operatorname{Lin}_{\mathrm{sym}}
\bigl((\mathbb R^d)^{\otimes j},E\bigr).
\]
Thus higher vertical derivatives are symmetric in their
vertical variables.
A non-anticipative functional
$G:\Lambda_T^d\to E$ is called boundedness-preserving if,
for every compact $K\subset\mathbb R^d$ and every
$t_0\in[0,T]$, there exists $C=C(K,t_0)$ such that whenever
$$
\forall t\in[0,t_0],
\qquad
X([0,t])\subset K\qquad{\rm we\  have}\quad \|G(t,X)\|_E\le C(K,t_0).$$
We denote this class by $B(\Lambda_T^d;E)$.

\subsection{Regularity concepts for non-anticipative functionals}\label{sec.regularity}
Regularity concepts for non-anticipative functionals have been developed in \cite{ananova2017,CF10B,cont2012,cont2019}. Here we recall some definitions and add some new ones which are tailored to the results in the sequel.
\begin{definition}
\label{def:C0p}
For $p\in\mathbb N$, we write $
F\in C_b^{0,p}(\Lambda_T^d;E)$
if $F$ is $p$ times vertically differentiable and
\[
\nabla_\omega^jF
\in
C_l^{0,0}
\bigl(
\Lambda_T^d;
\operatorname{Lin}_{\mathrm{sym}}
((\mathbb R^d)^{\otimes j},E)
\bigr)
\cap
B(\Lambda_T^d)
\]
for $j=0,\ldots,p$, with the convention
$\nabla_\omega^0F:=F$.
\end{definition}
\begin{definition}[Uniform supremum-norm Lipschitz continuity]
\label{def:sup-lip}
Let $G:\Lambda_T^d\to E$ be non-anticipative. We say that
$G$ is uniformly Lipschitz with respect to the supremum norm
if there exists $L_G^\infty<\infty$ such that
\begin{equation}
\|G(t,X)-G(t,X')\|_E
\le
L_G^\infty\|X_t-X'_t\|_\infty  \label{eq.uniformLipschitz}  
\end{equation}
for  $t\in[0,T]$ and 
$X,X'\in D([0,T],\mathbb R^d)$.
We write
$
[G]_{\mathrm{Lip}(\|\cdot\|_\infty)}
$
for the smallest constant $L_G^\infty$ satisfying \eqref{eq.uniformLipschitz}.
\end{definition}

\begin{definition}[Uniform horizontal Lipschitz continuity]
\label{def:horizontal-lip}
Let $G:\Lambda_T^d\to E$ be non-anticipative. We say that
$G$ is uniformly horizontally Lipschitz if there exists
$L_G^h<\infty$ such that
\begin{equation}
\|G(t+h,X_t)-G(t,X_t)\|_E
\le L_G^h h \label{eq.hLipschitz}    
\end{equation}
for every $X\in D([0,T],\mathbb R^d)$ and
$0\le t\le t+h\le T$.
We denote this class by
$
h\mathrm{Lip}(\Lambda_T^d;E)
$
and denote
$
[G]_{h\mathrm{Lip}}
$
for the smallest constant $L_G^\infty$ satisfying \eqref{eq.hLipschitz}.
\end{definition}
A non-anticipative functional $G:\Lambda_T^d\to E$ is
left-continuous if
\[
d_\infty\left((t_n,X^n_{t_n}),\ (t,X_t)\right)\mathop{\longrightarrow}^{t_n\uparrow t}_{t_n< t} 0\quad{\rm implies}\quad G(t_n,X^n_{t_n})\to G(t,X_t).
\]
We denote this class by $C_l^{0,0}(\Lambda_T^d;E)$.
We shall also need the following stronger notion.
\begin{definition}[Strong left continuity]
\label{def:strong-left-continuity}
We denote $
C_s^{0,0}(\Lambda_T^d;E).$ the class of non-anticipative functionals
$G:\Lambda_T^d\to E$ with the following property:
for every
$(t,X_t)\in\Lambda_T^d$ and every sequence
\[
(r_n,X^n_{r_n})\in\Lambda_T^d,
\qquad
r_n\le t,\qquad r_n\longrightarrow t,
\]
such that
\begin{eqnarray*}
\sup_n\|X^n_{r_n}\|_\infty<\infty,\qquad
X^n_{r_n}(u)\longrightarrow X_t(u),
\qquad u\in[0,t],
\\
\forall t'<t,\qquad 
\sup_{u\in[0,t']}
|X^n_{r_n}(u)-X_t(u)|
\longrightarrow0
\end{eqnarray*}
we have
\[
G(r_n,X^n_{r_n})
\longrightarrow
G(t,X_t).
\]
\end{definition}
Strong left continuity is used below to pass to limits along
continuous paths which approximate a vertical perturbation by
concentrating the perturbation on a shrinking interval immediately
preceding the current time.
\subsection{Controlled paths and rough integration}
\label{subsec:controlled-paths}

We finally fix the rough-path notation used in the sequel.
We use the standard notions of geometric rough path, controlled
path and rough integral; see
\cite{GUBINELLI200486,frizhairer,hairerkelly,lyonsqian} for details.

Let $\gamma\in(0,1)$ and set
\[
N:=\lfloor 1/\gamma\rfloor .
\]
Let $
\mathbf X
=
(1,\mathbf X^1,\ldots,\mathbf X^N)$
for a step-$N$ geometric $\gamma$-H\"older rough path over
$X\in C^\gamma([0,T],\mathbb R^d)$, with
\[
\mathbf X^1_{s,t}=X_{s,t}:=X(t)-X(s).
\]
Thus
\[
\|\mathbf X^j_{s,t}\|
\lesssim |t-s|^{j\gamma},
\qquad j=1,\ldots,N,\quad{\rm and}\quad
\mathbf X_{s,t}
=
\mathbf X_{s,u}\otimes\mathbf X_{u,t},
\qquad 0\le s\le u\le t\le T.
\]
We shall also use the standard identity
\begin{equation}
\operatorname{Sym}\mathbf X^j_{s,t}
=
\frac1{j!}X_{s,t}^{\otimes j},
\qquad j=1,\ldots,N.
\label{eq:sym-geometric-rough-path}
\end{equation}

For
\[
A\in
\operatorname{Lin}
\bigl((\mathbb R^d)^{\otimes(\ell+j)},E\bigr),
\qquad
z\in(\mathbb R^d)^{\otimes j},
\]
we denote by
\[
A[z]\in
\operatorname{Lin}
\bigl((\mathbb R^d)^{\otimes\ell},E\bigr)\qquad A[z_1,...z_j](v_1,...,v_{\ell})=A(v_1,...,v_{\ell},z_1,...z_j)
\]
the contraction of the last $j$ variables of $A$ against $z$.
In particular, if $A$ is symmetric in these $j$ variables,
then \eqref{eq:sym-geometric-rough-path} gives
\begin{equation}
A[\mathbf X^j_{s,t}]
=
\frac1{j!}A[X_{s,t}^{\otimes j}].
\label{eq:symmetric-contraction}
\end{equation}

\begin{definition}[Controlled path]
\label{def:controlled-path}
Let $\mathbf X$ be as above and let $E$ be a finite-dimensional
normed space. An $E$-valued path $Y^0$ is said to be controlled
by $\mathbf X$ if there exist paths
\[
Y^\ell:
[0,T]\longrightarrow
\operatorname{Lin}
\bigl((\mathbb R^d)^{\otimes\ell},E\bigr),
\qquad
\ell=1,\ldots,N-1,
\]
such that, with $Y^0:=Y$, the remainders
\begin{equation}
R^\ell_{s,t}
:=
Y^\ell_t-Y^\ell_s
-
\sum_{j=1}^{N-1-\ell}
Y^{\ell+j}_s[\mathbf X^j_{s,t}],
\qquad
\ell=0,\ldots,N-1,
\label{eq:controlled-remainder-definition}
\end{equation}
where the sum is understood to be empty for $\ell=N-1$,
satisfy
\begin{equation}
\|R^\ell_{s,t}\|
\lesssim
|t-s|^{(N-\ell)\gamma},
\qquad
\ell=0,\ldots,N-1.
\label{eq:standard-controlled-remainders}
\end{equation}
 $Y^j$ is known as the $j$-th Gubinelli coefficient (or "Gubinelli
derivative") of $Y$.
\end{definition}
This definition requires H\"older regularity which increases  linearly across levels: the
$j$-th rough-path level has H\"older regularity $j\gamma$, while the
remainder \eqref{eq:standard-controlled-remainders} at level $\ell$ has regularity
$(N-\ell)\gamma.$ We refer to this as the  {\it grading} property.
For example, when $\gamma\in(1/3,1/2]$, one has $N=2$ and
Definition~\ref{def:controlled-path} reduces to
\begin{equation}
Y_t-Y_s
=
Y'_s[X_{s,t}]
+
R^Y_{s,t},
\qquad
\|R^Y_{s,t}\|
\lesssim |t-s|^{2\gamma},
\label{eq:first-order-controlled-path}
\end{equation}
with
\[
\|Y'_t-Y'_s\|
\lesssim |t-s|^\gamma.
\]
This is the ``classical controlled remainder'' referred to
below.

The expansions obtained below for path-dependent functionals
have the same Taylor-type algebraic structure, but in general
their remainder exponents are not given by this linear grading \cite{ananova2017,ananova2023,Bielert2026}.

Finally, we recall the {\it Sewing Lemma} \cite[Lemma 4.2]{frizhairer}. For a
two-parameter map $\Xi:\Delta_T\to E$, set
\[
\delta\Xi_{s,u,t}
:=
\Xi_{s,t}-\Xi_{s,u}-\Xi_{u,t}.
\]
If, for some $\theta>1$,
\[
\|\delta\Xi_{s,u,t}\|
\lesssim |t-s|^\theta,
\]
then there exists a unique additive map $I:\Delta_T\to E$
such that
\[
I_{s,t}
=
\lim_{|\pi|\to0}
\sum_{[u,v]\in\pi}\Xi_{u,v}\qquad{\rm and}\qquad \|I_{s,t}-\Xi_{s,t}\|
\lesssim |t-s|^\theta.
\]
In particular, if $Y$ is a controlled one-form with
coefficients $Y^0,\ldots,Y^{N-1}$, the rough integral is
obtained by sewing
\begin{equation}
\Xi_{s,t}
=
\sum_{\ell=0}^{N-1}
Y^\ell_s[\mathbf X^{\ell+1}_{s,t}],
\label{eq:controlled-rough-integral-approximant}
\end{equation}
where we have used the  identification
\[
\operatorname{Lin}
\bigl(
(\mathbb R^d)^{\otimes\ell},
\operatorname{Lin}(\mathbb R^d,E)
\bigr)
\simeq
\operatorname{Lin}
\bigl(
(\mathbb R^d)^{\otimes(\ell+1)},E
\bigr).
\]
\section{Regular functionals as controlled paths} \label{sec.expansion}

The first-order relation between functional derivatives and
controlled paths recalled in Section~\ref{subsec:controlled-paths} raises two 
questions. First, under the above assumptions, does the hierarchy of vertical derivatives admit a
higher-order controlled expansion along a H\"older path? 
Conversely, when does a compatible family of controlled
coefficients coincide with the  vertical
derivatives of its base component?

We address the first question in this section.
We  study the direct passage from non-anticipative
functional calculus to controlled rough paths. Under weak
horizontal regularity, we establish higher-order controlled
Taylor expansions whose coefficients are the vertical
derivatives. We then determine when these estimates are
sufficient for rough integration and identify a
regime in which the classical controlled
remainder is recovered.

The approximation argument developed below originates in the
first-order estimate of ~\cite[Lemma~2.2]{ananova2017},
where a H\"older path is approximated by paths of bounded
variation and the approximation scale is optimized by balancing
the resulting error terms.  Bielert~\cite{Bielert2026} recently
extended this strategy to higher-order functional Taylor
expansions under stronger horizontal differentiability assumptions.
Here, in contrast to \cite{Bielert2026}, no horizontal (time) differentiability
 nor any mixed differentiability e.g. $D\nabla_\omega^jF$, is required.
The resulting loss in the remainder exponent quantifies the
price paid for dispensing with  horizontal differentiability.

\subsection{Functional Taylor expansion under weak horizontal regularity}
\label{sec:weak-horizontal-Taylor}

We now establish the higher-order counterpart of the
first-order estimate recalled in Section~\ref{sec:functional-calculus}. Throughout this
subsection, horizontal regularity is assumed only in the
Lipschitz sense introduced above; in particular, no 
horizontal/time derivatives are required.

We will  use  the following lemma which gives us a convenient way to express the changes of a regular non-anticipative functional on a piecewise constant path. 

\begin{lemma}[Taylor estimate on piecewise-constant paths]
\label{lem:pc-taylor}
Let $p\in\mathbb N$, let
\[
E_j:=\operatorname{Lin}_{\mathrm{sym}}
\bigl((\mathbb R^d)^{\otimes j},\mathbb R^v\bigr),
\qquad E_0:=\mathbb R^v,
\]
and $
F\in C_b^{0,p}(\Lambda_T^d;\mathbb R^v)$ with
\[
\nabla_\omega^jF \in h\operatorname{Lip}(\Lambda_T^d;E_j)
\cap
\operatorname{Lip}(\Lambda_T^d,\|\cdot\|_\infty;E_j),
\qquad j=0,\ldots,p.
\]
Let $0\le s<t\le T$ and assume  $Y\in D([0,T],\mathbb R^d)$ is piecewise constant on $\pi=\{s=t_0<t_1<\cdots<t_N=t\}.$
Let
\[
z_i:=Y(t_{i+1})-Y(t_i),
\qquad
V:=\sum_{i=0}^{N-1}|z_i|,
\qquad 
\Omega:=\sup_{r\in[s,t]}|Y(r)-Y(s)|.
\]
Denoting $G_j:=\nabla_\omega^jF$, $j=0,\ldots,p$, for every $\ell=0,\ldots,p$, with $n:=p-\ell$, we have
\[
G_\ell(t,Y_t)-G_\ell(s,Y_s)
=
\sum_{j=1}^{n}
\frac1{j!}
G_{\ell+j}(s,Y_s)
[Y_{s,t}^{\otimes j}]
+
\mathcal E^\ell_{s,t},
\]
where the sum is empty if $\ell=p$, and
\[
\|\mathcal E^\ell_{s,t}\|_{E_\ell}
\le
C_{\ell,p}
\left(
|t-s|\sum_{m=0}^{n}V^m
+
\Omega V^n
\right)
\]
where $C_{\ell,p}$ depends only on $p,\ell$ and on the horizontal
and sup-norm Lipschitz constants of
$G_\ell,\ldots,G_p$.
\end{lemma}
\begin{proof}
Set
$
G_j:=\nabla_\omega^jF,
j=0,\ldots,p,$
with the convention \(G_0:=F\). For \(i=0,\ldots,N-1\), define
$
z_i:=Y(t_{i+1})-Y(t_i),
$
and, for \(i=0,\ldots,N\),
\[
S_i:=Y(t_i)-Y(s)=\sum_{r=0}^{i-1}z_r,
\qquad
V_i:=\sum_{r=0}^{i-1}|z_r|,
\]
with \(S_0=0\) and \(V_0=0\). Thus
\[
S_N=Y(t)-Y(s)=Y_{s,t},
\qquad
V_N=V.
\]
We also set $
\Omega_i:=\sup_{r\in[s,t_i]}|Y(r)-Y(s)|,$
so that
\[
|S_i|\le \Omega_i\le \Omega,
\qquad
V_i\le V.
\]
We prove the assertion by induction on
$
n:=p-\ell.
$
Suppose first that \(n=0\), i.e. \(\ell=p\). Then
\[
\begin{aligned}
G_p(t,Y_t)-G_p(s,Y_s)
=&
G_p(t,Y_t)-G_p(t,Y_s)
+
G_p(t,Y_s)-G_p(s,Y_s).
\end{aligned}
\]
Since
$
\|Y_t-Y_s\|_\infty\le \Omega,$
the supremum-norm Lipschitz continuity of \(G_p\) gives
\[
\|G_p(t,Y_t)-G_p(t,Y_s)\|
\le L_p^\infty \Omega,
\]
whereas the horizontal Lipschitz continuity gives
\[
\|G_p(t,Y_s)-G_p(s,Y_s)\|
\le L_p^h|t-s|.
\]
Hence
\[
\|G_p(t,Y_t)-G_p(s,Y_s)\|
\le
L_p^h|t-s|+L_p^\infty\Omega,
\]
which proves the assertion for \(n=0\).
Now let \(n\ge1\), set \(\ell=p-n\), and assume that the result holds at the levels
$
\ell+1,\ldots,p.
$
Thus, for every \(j=1,\ldots,n\) and every \(i=0,\ldots,N\),
\begin{equation}
\begin{aligned}
G_{\ell+j}(t_i,Y_{t_i})
={}&
\sum_{r=0}^{n-j}
\frac1{r!}
G_{\ell+j+r}(s,Y_s)
[S_i^{\otimes r}]
+
E_i^{\ell+j},
\end{aligned}
\label{eq.1}
\end{equation}
where
\begin{equation}
\|E_i^{\ell+j}\|
\le
C
\left(
(t_i-s)\sum_{m=0}^{n-j}V_i^m
+
\Omega_iV_i^{n-j}
\right).
\label{eq.2}
\end{equation}

For each \(i=0,\ldots,N-1\), write
\begin{equation}
\begin{aligned}
&G_\ell(t_{i+1},Y_{t_{i+1}})
-
G_\ell(t_i,Y_{t_i})
=
\underbrace{
G_\ell(t_{i+1},Y_{t_i})
-
G_\ell(t_i,Y_{t_i})
}_{=:H_i^\ell}
+
\underbrace{
G_\ell(t_{i+1},Y_{t_{i+1}})
-
G_\ell(t_{i+1},Y_{t_i})
}_{=:V_i^\ell}.
\end{aligned}
\label{eq.3}
\end{equation}
By horizontal Lipschitz continuity,
\begin{equation}
\|H_i^\ell\|
\le
L_\ell^h(t_{i+1}-t_i).
\label{eq.4}
\end{equation}

Since \(Y\) is constant on \([t_i,t_{i+1})\), the stopped paths
\(Y_{t_i}\) and \(Y_{t_{i+1}}\), when viewed at time \(t_{i+1}\),
differ by the vertical perturbation
$
z_i\mathbf 1_{[t_{i+1},T]}.$
Hence, defining
\[
\phi_i(h)
:=
G_\ell
\bigl(
t_{i+1},
Y_{t_i}+hz_i\mathbf 1_{[t_{i+1},T]}
\bigr),
\qquad h\in[0,1],
\]
we have, for \(j=1,\ldots,n\),
\[
\phi_i^{(j)}(h)
=
G_{\ell+j}
\bigl(
t_{i+1},
Y_{t_i}+hz_i\mathbf 1_{[t_{i+1},T]}
\bigr)
[z_i^{\otimes j}].
\]
By Taylor's formula,
\begin{equation}
V_i^\ell
=
\sum_{j=1}^{n}
\frac1{j!}
G_{\ell+j}(t_{i+1},Y_{t_i})
[z_i^{\otimes j}]
+
r_i,
\label{eq.5}
\end{equation}
where, using the supremum-norm Lipschitz continuity of \(G_p\),
\begin{equation}
\|r_i\|
\le
C|z_i|^{n+1}.
\label{eq.6}
\end{equation}

Next, for \(j=1,\ldots,n\), write
\begin{equation}
G_{\ell+j}(t_{i+1},Y_{t_i})
=
G_{\ell+j}(t_i,Y_{t_i})
+
K_{i,j},
\qquad 
\|K_{i,j}\|
\le
L_{\ell+j}^h(t_{i+1}-t_i).
\label{eq.7}
\end{equation}
Substituting \eqref{eq.7} and then the induction hypothesis
\eqref{eq.1} into \eqref{eq.5}, we obtain
\begin{equation}
\begin{aligned}
V_i^\ell
={}&
\sum_{j=1}^{n}
\frac1{j!}
\sum_{r=0}^{n-j}
\frac1{r!}
G_{\ell+j+r}(s,Y_s)
[S_i^{\otimes r},z_i^{\otimes j}]
\\
&+
\sum_{j=1}^{n}
\frac1{j!}
E_i^{\ell+j}[z_i^{\otimes j}]
+
\sum_{j=1}^{n}
\frac1{j!}
K_{i,j}[z_i^{\otimes j}]
+
r_i.
\end{aligned}
\label{eq.9}
\end{equation}
We now sum over \(i\). Consider first the principal part, namely
the first line of \eqref{eq.9}. For fixed
\(q\in\{1,\ldots,n\}\), the contribution of all terms with
\(j+r=q\) is
\begin{equation}
\sum_{i=0}^{N-1}
\sum_{j=1}^{q}
\frac1{j!(q-j)!}
G_{\ell+q}(s,Y_s)
[S_i^{\otimes(q-j)},z_i^{\otimes j}].
\label{eq.10}
\end{equation}
Since \(G_{\ell+q}(s,Y_s)\) is symmetric in the contracted
vertical variables and
$
S_{i+1}=S_i+z_i,
$
the binomial identity gives
\begin{equation}
\begin{aligned}
&\sum_{j=1}^{q}
\frac1{j!(q-j)!}
G_{\ell+q}(s,Y_s)
[S_i^{\otimes(q-j)},z_i^{\otimes j}]
=
\frac1{q!}
G_{\ell+q}(s,Y_s)
\bigl[
S_{i+1}^{\otimes q}-S_i^{\otimes q}
\bigr].
\end{aligned}
\label{eq.11}
\end{equation}
Summing \eqref{eq.11} over \(i\) therefore yields
\[
\frac1{q!}
G_{\ell+q}(s,Y_s)
[S_N^{\otimes q}],
\]
because \(S_0=0\). Since \(S_N=Y_{s,t}\), the principal terms
sum exactly to
\begin{equation}
\sum_{q=1}^{n}
\frac1{q!}
G_{\ell+q}(s,Y_s)
[Y_{s,t}^{\otimes q}].
\label{eq.12}
\end{equation}

It remains to estimate the error terms. By \eqref{eq.4},
\begin{equation}
\sum_{i=0}^{N-1}\|H_i^\ell\|
\le
L_\ell^h|t-s|.
\label{eq.13}
\end{equation}
Using \eqref{eq.7},
\begin{equation}
\begin{aligned}
\sum_{i=0}^{N-1}
\sum_{j=1}^{n}
\|K_{i,j}[z_i^{\otimes j}]\|
&\le
C
\sum_{i=0}^{N-1}
(t_{i+1}-t_i)
\sum_{j=1}^{n}|z_i|^j
\le
C|t-s|
\sum_{j=1}^{n}V^j.
\end{aligned}
\label{eq.14}
\end{equation}
We next estimate the terms involving the induction remainders.
By \eqref{eq.2},
\begin{equation}
\begin{aligned}
&\sum_{i=0}^{N-1}
\sum_{j=1}^{n}
\|E_i^{\ell+j}\|\,|z_i|^j
\le
C
\sum_{j=1}^{n}
\sum_{i=0}^{N-1}
(t_i-s)
\left(
\sum_{m=0}^{n-j}V_i^m
\right)|z_i|^j
+
C
\sum_{j=1}^{n}
\sum_{i=0}^{N-1}
\Omega_iV_i^{n-j}|z_i|^j.
\end{aligned}
\label{22}
\end{equation}
For non-negative numbers \(a_i:=|z_i|\), we have
$$
V_i=\sum_{r<i}a_r,
\qquad
V=\sum_i a_i,
$$
and therefore, for \(m\ge0\) and \(j\ge1\),
\begin{equation}
\sum_i V_i^m a_i^j
\le
V^{m+j}.
\label{23}
\end{equation}
Indeed,
\[
V_i^m\le V^m,
\qquad
a_i^j\le V^{j-1}a_i,
\]
and hence
\[
\sum_iV_i^ma_i^j
\le
V^{m+j-1}\sum_i a_i
=
V^{m+j}.
\]
Since \(t_i-s\le t-s\), \eqref{23} implies
\begin{equation}
\begin{aligned}
&\sum_{j=1}^{n}
\sum_{i=0}^{N-1}
(t_i-s)
\left(
\sum_{m=0}^{n-j}V_i^m
\right)|z_i|^j \le
C|t-s|
\sum_{j=1}^{n}
\sum_{m=0}^{n-j}
V^{m+j}
\le
C|t-s|
\sum_{q=1}^{n}V^q.
\end{aligned}
\label{24}
\end{equation}
Similarly, since \(\Omega_i\le\Omega\),
\begin{equation}
\sum_{j=1}^{n}
\sum_{i=0}^{N-1}
\Omega_iV_i^{n-j}|z_i|^j
\le
C\Omega V^n.
\label{25}
\end{equation}
Combining \eqref{22}-\eqref{24}-\eqref{25},
\begin{equation}
\sum_{i=0}^{N-1}
\sum_{j=1}^{n}
\|E_i^{\ell+j}\|\,|z_i|^j
\le
C
\left(
|t-s|\sum_{q=1}^{n}V^q
+
\Omega V^n
\right).
\label{26}
\end{equation}
Finally, by \eqref{eq.6},
\begin{equation}
\sum_{i=0}^{N-1}\|r_i\|
\le
C\sum_{i=0}^{N-1}|z_i|^{n+1}.
\label{eq.20}
\end{equation}
Since
$
z_i=S_{i+1}-S_i,$
we have
$
|z_i|
\le
|S_{i+1}|+|S_i|
\le
2\Omega.$
Hence
\begin{equation}
\sum_{i=0}^{N-1}|z_i|^{n+1}
\le
2\Omega\sum_{i=0}^{N-1}|z_i|^n
\le
2\Omega V^n,
\label{eq.21}
\end{equation}
and consequently
\begin{equation}
\sum_{i=0}^{N-1}\|r_i\|
\le
C\Omega V^n.
\label{29}
\end{equation}

Summing \eqref{eq.3} over \(i\), using the exact identity
\eqref{eq.12}, and combining the bounds
\eqref{eq.13}, \eqref{eq.14}, \eqref{26}, and \eqref{29}, we obtain
\[
\begin{aligned}
&\left\|
G_\ell(t,Y_t)-G_\ell(s,Y_s)
-
\sum_{q=1}^{n}
\frac1{q!}
G_{\ell+q}(s,Y_s)
[Y_{s,t}^{\otimes q}]
\right\|
\le
C_{\ell,p}
\left(
|t-s|\sum_{m=0}^{n}V^m
+
\Omega V^n
\right).
\end{aligned}
\]
This proves the induction step and completes the proof.
\end{proof}

Using the results of the previous two lemmas, we can now prove the following Taylor expansion for non-anticipative functionals.
\begin{proposition}[Higher-order functional Taylor expansion]
\label{prop:functional-taylor-expansion}
Let $p\in\mathbb{N}$ and $\gamma\in(0,1)$. For
$j=0,\ldots,p$, set
\[
E_j
:=
\operatorname{Lin}_{\mathrm{sym}}
\big((\mathbb{R}^d)^{\otimes j},\mathbb{R}^v\big),
\qquad
E_0:=\mathbb{R}^v.
\]
Let
$F\in C_b^{0,p}(\Lambda_T^d,\mathbb R^v),$
and, with the convention $\nabla_\omega^0F:=F$, assume that
for $j=0,\ldots,p$,
\[
\nabla_\omega^jF
\in
h\operatorname{Lip}(\Lambda_T^d;E_j)
\cap
\operatorname{Lip}
(\Lambda_T^d,\|\cdot\|_\infty;E_j).
\]
Denote the corresponding Lipschitz constants by
$
L_j^h
:=
[\nabla_\omega^jF]_{h\operatorname{Lip}},
L_j^\infty
:=
[\nabla_\omega^jF]_{\operatorname{Lip}(\|\cdot\|_\infty)}.$\\
For $A\in E_{\ell+j}$ and
$h\in\mathbb{R}^d$, denote $
A[h^{\otimes j}]\in E_\ell $  the partial contraction
\[
A[h^{\otimes j}]
(v_1,\ldots,v_\ell)
:=
A(v_1,\ldots,v_\ell,
  \underbrace{h,\ldots,h}_{j\text{ times}}).
\]

Let $X\in C^\gamma([0,T],\mathbb{R}^d)$.
For $\ell\in\{0,\ldots,p-1\}$ denote
\[
n_\ell:=p-\ell,
\quad
D_\ell
:=
n_\ell(1-\gamma)+\gamma,\quad
q_\ell
:=
\frac{n_\ell+\gamma}{D_\ell},
\quad
\beta_\ell
:=
\gamma q_\ell
=
\frac{\gamma(n_\ell+\gamma)}
     {n_\ell(1-\gamma)+\gamma}.
\]
Then $q_\ell>1$ and, for every $0\le s<t\le T$
such that
$
\rho_{s,t}
:=
\|X\|_{\gamma;[s,t]}|t-s|^\gamma
\le 1,$
we have
\begin{align}
\nabla_\omega^\ell F(t,X_t)
-
\nabla_\omega^\ell F(s,X_s)
&=
\sum_{j=1}^{p-\ell}
\frac{1}{j!}
\nabla_\omega^{\ell+j}F(s,X_s)
[X_{s,t}^{\otimes j}]
+
R^\ell_{s,t},
\label{eq:functional-taylor-expansion}
\end{align}
where
\begin{equation}
\|R^\ell_{s,t}\|_{E_\ell}
\le
C_{\ell,F,T}
\left(
|t-s|
+
\|X\|_{\gamma;[s,t]}^{q_\ell}
|t-s|^{\beta_\ell}
\right)=C_{\ell,F,T}
\left(
|t-s|+\rho_{s,t}^{q_\ell}
\right).
\label{eq:functional-Taylor-uniform-remainder}
\end{equation}
Here $C_{\ell,F,T}$ is independent of $X,s,t$ and may be
chosen to depend only on $p,\gamma,d,v,T$ and on the
Lipschitz constants
$
\{L_j^h,L_j^\infty:0\le j\le p\}.$
For $\ell=p$, one has
\[
\|\nabla_\omega^pF(t,X_t)-\nabla_\omega^pF(s,X_s)\|_{E_p}
\le
L_p^h|t-s|
+
L_p^\infty
\|X\|_{\gamma;[s,t]}|t-s|^\gamma .
\]
In particular,
\[
\big(
F(\cdot,X),
\nabla_\omega F(\cdot,X),
\ldots,
\nabla_\omega^pF(\cdot,X)
\big)
\]
is a controlled family with respect to $X$,
with remainder exponents $\beta_\ell$,
$\ell=0,\ldots,p-1$.
\end{proposition}

\begin{proof}
Fix $\ell\in\{0,\ldots,p-1\}$ and write
\[
n:=p-\ell,
\qquad
\delta:=t-s,
\qquad
A:=\|X\|_{\gamma;[s,t]},
\qquad
\rho:=A\delta^\gamma.
\]
If $\rho=0$, then $X$ is constant on $[s,t]$, and the assertion
follows immediately from the horizontal Lipschitz continuity of
$\nabla_\omega^\ell F$. Hence we may assume
$
0<\rho\le1.$
Let $
\pi_N=\{t_0,\ldots,t_N\},
t_i:=s+\frac{i}{N}(t-s),$
be the uniform partition of $[s,t]$, and define the
piecewise-constant approximation
\[
X^N(r):=
\begin{cases}
X(r), & r<s,\\
X(t_i), & t_i\le r<t_{i+1},
             \quad i=0,\ldots,N-1,\\
X(r), & r\ge t.
\end{cases}
\]
Then
\[
X^N_s=X_s,
\qquad
X^N(t)=X(t),
\qquad
X^N_{s,t}=X_{s,t},
\]
and
\begin{equation}
\|X^N_t-X_t\|_\infty
\le
A\delta^\gamma N^{-\gamma}
=
\rho N^{-\gamma}.
\label{eq.15}
\end{equation}

Define
\begin{eqnarray}
R^\ell_{s,t}
:=
\nabla_\omega^\ell F(t,X_t)
-\nabla_\omega^\ell F(s,X_s)
-
\sum_{j=1}^{n}
\frac1{j!}
\nabla_\omega^{\ell+j}F(s,X_s)
[X_{s,t}^{\otimes j}],
\label{eq.16}\\
R^{\ell,N}_{s,t}
:=
\nabla_\omega^\ell F(t,X^N_t)
-\nabla_\omega^\ell F(s,X^N_s)
-
\sum_{j=1}^{n}
\frac1{j!}
\nabla_\omega^{\ell+j}F(s,X^N_s)
[(X^N_{s,t})^{\otimes j}].
\label{eq.17}
\end{eqnarray}
Since $
X^N_s=X_s,
X^N_{s,t}=X_{s,t},$
the Taylor polynomials in \eqref{eq.16} and \eqref{eq.17} coincide.
Therefore, by the supremum-norm Lipschitz continuity of
$\nabla_\omega^\ell F$,
\begin{equation}
\|R^\ell_{s,t}-R^{\ell,N}_{s,t}\|_{E_\ell}
\le
L_\ell^\infty \rho N^{-\gamma}.
\label{eq.18}
\end{equation}
We now apply Lemma~\ref{lem:pc-taylor} to $X^N$ on $[s,t]$.
Let
\[
V_N
:=
\sum_{i=0}^{N-1}
|X(t_{i+1})-X(t_i)|,\qquad
\Omega_N
:=
\sup_{r\in[s,t]}
|X^N(r)-X(s)|.
\]
By the $\gamma$-H\"older continuity of $X$,
\[
|X(t_{i+1})-X(t_i)|
\le
A\left(\frac{\delta}{N}\right)^\gamma,
\]
and hence
\begin{equation}
V_N
\le
A\delta^\gamma N^{1-\gamma}
=
\rho N^{1-\gamma}.
\label{eq.19}
\end{equation}
Moreover,
\begin{equation}
\Omega_N
\le
A\delta^\gamma
=
\rho.
\label{eq.20}
\end{equation}

Lemma~\ref{lem:pc-taylor} therefore yields
\begin{equation}
\begin{aligned}
\|R^{\ell,N}_{s,t}\|_{E_\ell}
&\le
C_{\ell,p}
\left(
\delta\sum_{m=0}^{n}V_N^m
+
\Omega_NV_N^n
\right)
\\
&\le
C
\left[
\delta
+
\sum_{m=1}^{n}
\delta\,\rho^mN^{m(1-\gamma)}
+
\rho^{n+1}N^{n(1-\gamma)}
\right].
\end{aligned}
\label{eq.21}
\end{equation}
Combining this with \eqref{eq.18}, we obtain
\[
\|R^\ell_{s,t}\|_{E_\ell}
\le
C
\left[
\delta
+
\sum_{m=1}^{n}
\delta\,\rho^mN^{m(1-\gamma)}
+
\rho^{n+1}N^{n(1-\gamma)}
+
\rho N^{-\gamma}
\right].
\label{eq.22}
\]

We now optimize the choice of $N$. The two terms which determine
the relevant scale are
\[
\rho^{n+1}N^{n(1-\gamma)}
\qquad\text{and}\qquad
\rho N^{-\gamma}.
\]
Balancing them gives
$
N^{n(1-\gamma)+\gamma}\asymp \rho^{-n}.$
Set
$
D:=n(1-\gamma)+\gamma
$
and choose
$
N:=\left\lceil \rho^{-n/D}\right\rceil .
\label{eq.23}
$
Since $\rho\le1$,
$
N\asymp \rho^{-n/D}.
$
For the approximation term,
\[
\rho N^{-\gamma}
\le
C\rho^{1+n\gamma/D}
=
C\rho^{(n+\gamma)/D}.
\label{eq.24}
\]
Likewise,
\[
\rho^{n+1}N^{n(1-\gamma)}
\le
C
\rho^{n+1-n^2(1-\gamma)/D}
=
C\rho^{(n+\gamma)/D}.
\label{eq.25}
\]
Finally,  since $0<\rho\le1$ we have for $1\le m\le n$,
\begin{equation}
\begin{aligned}
\delta\,\rho^mN^{m(1-\gamma)}
&\le
C\delta\,
\rho^{m-mn(1-\gamma)/D}
=
C\delta\,\rho^{m\gamma/D}
\le
C\delta.
\end{aligned}
\label{eq.26}
\end{equation}
Thus
\[
\|R^\ell_{s,t}\|_{E_\ell}
\le
C_{\ell,F,T}
\left(
\delta+\rho^{q_\ell}
\right),
\]
where
\[
q_\ell
=
\frac{n+\gamma}{D}
=
\frac{p-\ell+\gamma}
{(p-\ell)(1-\gamma)+\gamma}.
\]
Since
$
\rho^{q_\ell}
=
\|X\|_{\gamma;[s,t]}^{q_\ell}
|t-s|^{\gamma q_\ell},
$
this is precisely
\[
\|R^\ell_{s,t}\|_{E_\ell}
\le
C_{\ell,F,T}
\left(
|t-s|
+
\|X\|_{\gamma;[s,t]}^{q_\ell}
|t-s|^{\beta_\ell}
\right),
\]
with
\[
\beta_\ell
=
\gamma q_\ell
=
\frac{\gamma(p-\ell+\gamma)}
{(p-\ell)(1-\gamma)+\gamma}.
\]

It remains to consider $\ell=p$. In this case there is no Taylor
polynomial to subtract. By horizontal and supremum-norm Lipschitz
continuity,
\[
\begin{aligned}
\|\nabla_\omega^pF(t,X_t)
-\nabla_\omega^pF(s,X_s)\|_{E_p}
&\le
\|\nabla_\omega^pF(t,X_t)
-\nabla_\omega^pF(t,X_s)\|_{E_p}
+
\|\nabla_\omega^pF(t,X_s)
-\nabla_\omega^pF(s,X_s)\|_{E_p}
\\
&\le
L_p^\infty
\|X_t-X_s\|_\infty
+
L_p^h|t-s|
\\
&\le
L_p^\infty
\|X\|_{\gamma;[s,t]}|t-s|^\gamma
+
L_p^h|t-s|.
\end{aligned}
\]
This proves the result.
\end{proof}

\begin{remark}
\label{rem:comparison-AC-Bielert}
The proof of Proposition~\ref{prop:functional-taylor-expansion} uses the approximation scheme introduced in
~\cite[Lemma~2.2]{ananova2017}. Their proof
approximates the H\"older control by bounded-variation paths,
transfers functional estimates from the approximants to the
original path through supremum-norm continuity, and optimizes
the approximation scale, yielding the exponent
$\gamma(1+\gamma)$ at first order. 

A higher-order variant of this argument was
developed by Bielert~\cite[Corollary~2.3]{Bielert2026}  under  stronger
regularity assumptions involving the causal time derivatives.
This permits
stronger sewing estimates, whereas Proposition \ref{prop:functional-taylor-expansion} retains
only horizontal Lipschitz continuity and obtains  level-dependent
remainder exponents
\[
\beta_\ell
=
\frac{\gamma(p-\ell+\gamma)}
     {(p-\ell)(1-\gamma)+\gamma}.
\]
\end{remark}
Comparing Proposition \ref{prop:functional-taylor-expansion} with  Definition \ref{def:controlled-path}, we see that the vertical-derivative hierarchy
\[
\bigl(
F(\cdot,X),
\nabla_\omega F(\cdot,X),
\ldots,
\nabla_\omega^pF(\cdot,X)
\bigr)
\]
has the algebraic structure of a higher-order controlled
expansion, but with the level-dependent remainder exponents
\[
\beta_\ell
=
\frac{\gamma(p-\ell+\gamma)}
{(p-\ell)(1-\gamma)+\gamma},
\qquad \ell=0,\ldots,p-1,
\]
rather than the standard linear grading.
The consequences of these exponents for rough integration are
examined in the next subsection.
\subsection{Rough integration under weak horizontal regularity}
\label{sec:weak-horizontal-integration}

We next examine when the expansion of
Proposition~\ref{prop:functional-taylor-expansion} is by itself
sufficient to define a rough integral. The result below is a
consequence of the Sewing Lemma and should be interpreted as
a consequence of the weak horizontal regularity assumptions,
rather than as a general limitation on rough integration of
path-dependent functionals.

\begin{theorem}[Rough integration under weak horizontal regularity]
\label{thm:weak-horizontal-rough-integral}
Let $\gamma\in(0,1/2]$ and $
p:=\lfloor1/\gamma\rfloor.$
Let
$
\mathbf X
=
(1,\mathbf X^{1},\ldots,\mathbf X^{p})$
be a geometric $\gamma$-H\"older rough path over
$X=\mathbf X^1$ and
\[
F:\Lambda_T^d
\longrightarrow
\operatorname{Lin}(\mathbb R^d,\mathbb R^v)
\]
a non-anticipative functional satisfying the assumptions of
Proposition~\ref{prop:functional-taylor-expansion} up to vertical order
$p$.
For $j=0,\ldots,p-1$, we
regard
\[
\nabla_\omega^jF(s,X_s)
\]
as a $(j+1)$-linear map, the final tensor variable being the
one-form variable, and set
\[
\Xi_{s,t}
:=
\sum_{j=0}^{p-1}
\left\langle
\nabla_\omega^jF(s,X_s),
\mathbf X^{j+1}_{s,t}
\right\rangle .
\]

For $\ell=0,\ldots,p-2$, let
\[
\eta_\ell
:=
\beta_\ell+(\ell+1)\gamma,
\qquad
\beta_\ell
=
\frac{\gamma(p-\ell+\gamma)}
     {(p-\ell)(1-\gamma)+\gamma}.
\]
If $
\min_{0\le\ell\le p-2}\eta_\ell>1,$
there exists a unique additive map
$ I:\Delta_T\to\mathbb R^v $
such that, locally in $(s,t)$,
\[
\|I_{s,t}-\Xi_{s,t}\|
\le
C|t-s|^\theta
\]
for some $\theta>1$.
In particular,
\[
\int_s^t F(r,X_r)\,d\mathbf X_r
:=
I_{s,t}
=
\lim_{|\pi|\to0}
\sum_{[u,v]\in\pi}\Xi_{u,v}
\]
is well defined.
\end{theorem}
\begin{proof}
Write
\[
Y^\ell_t:=\nabla_\omega^\ell F(t,X_t),
\qquad
\ell=0,\ldots,p-1.
\]
For $\ell\le p-2$, define the remainder appropriate to a
step-$p$ controlled expansion by
\[
\overline R^\ell_{s,t}
:=
Y^\ell_t-Y^\ell_s
-
\sum_{j=1}^{p-1-\ell}
\frac1{j!}
\nabla_\omega^{\ell+j}F(s,X_s)
[X_{s,t}^{\otimes j}].
\]
By Proposition~\ref{prop:functional-taylor-expansion},
\[
\overline R^\ell_{s,t}
=
\frac1{(p-\ell)!}
\nabla_\omega^pF(s,X_s)
[X_{s,t}^{\otimes(p-\ell)}]
+
R^\ell_{s,t}.
\]
Consequently, on sufficiently small intervals,
\[
\|\overline R^\ell_{s,t}\|
\lesssim
|t-s|
+
|t-s|^{\beta_\ell}
+
|t-s|^{(p-\ell)\gamma}.
\]
Moreover, the estimate at level $p-1$ gives
\[
\|Y^{p-1}_t-Y^{p-1}_s\|
\lesssim |t-s|^\gamma.
\]
Since $\nabla_\omega^{\ell+j}F(s,X_s)$ is symmetric in
the $j$ vertical variables being contracted, and since for a
geometric rough path
\[
\operatorname{Sym}\mathbf X^j_{s,t}
=
\frac1{j!}X_{s,t}^{\otimes j},
\]
the Taylor term
\[
\frac1{j!}
\nabla_\omega^{\ell+j}F(s,X_s)
[X_{s,t}^{\otimes j}]
\]
may equivalently be written as the contraction of
$\nabla_\omega^{\ell+j}F(s,X_s)$ with
$\mathbf X^j_{s,t}$ in these $j$ vertical variables.
Hence Proposition~\ref{prop:functional-taylor-expansion} provides precisely the controlled
expansions required in the Chen-relation computation for
$\Xi$.
The standard cancellation gives
\[
|\delta\Xi_{s,u,t}|
\lesssim
\sum_{\ell=0}^{p-2}
\|\overline R^\ell_{s,u}\|
\,
\|\mathbf X^{\ell+1}_{u,t}\|
+
\|Y^{p-1}_u-Y^{p-1}_s\|
\,
\|\mathbf X^p_{u,t}\|.
\]
It follows that
\[
|\delta\Xi_{s,u,t}|
\lesssim
\sum_{\ell=0}^{p-2}
\left(
|t-s|^{1+(\ell+1)\gamma}
+
|t-s|^{\beta_\ell+(\ell+1)\gamma}
+
|t-s|^{(p+1)\gamma}
\right)
+
|t-s|^{(p+1)\gamma}.
\]
Now
\[
1+(\ell+1)\gamma>1,
\qquad
(p+1)\gamma>1.
\]
Thus the only non-automatic condition is
\[
\beta_\ell+(\ell+1)\gamma>1,
\qquad
\ell=0,\ldots,p-2.
\]
The Sewing Lemma yields the result.
\end{proof}

\begin{corollary}[Threshold supplied by Proposition~\ref{prop:functional-taylor-expansion}]
\label{cor:two-fifths}
Let
\[
p=\lfloor1/\gamma\rfloor,
\qquad
\gamma\in(0,1/2].
\]
Then the quantities
$
\eta_\ell
=
\beta_\ell+(\ell+1)\gamma$
are strictly increasing in $\ell$. Hence the restrictive
condition in
Theorem~\ref{thm:weak-horizontal-rough-integral}
is
\[
\beta_0+\gamma>1.
\]

For $\gamma\in(1/3,1/2]$, where $p=2$, this is equivalent to
\[
\frac{\gamma(2+\gamma)}{2-\gamma}+\gamma>1,
\]
and therefore to
\[
\gamma>\frac25.
\]
For $\gamma\le1/3$ the same condition fails. Consequently the
estimate of Proposition~\ref{prop:functional-taylor-expansion}, used
without additional structure, yields rough integration in the
range
$
\gamma>\frac25.$
\end{corollary}

 \begin{remark}{\em 
The threshold $\gamma>2/5$ is therefore not an intrinsic
barrier for rough functional integration. It is the threshold
obtained when the Ananova--Cont approximation strategy is
carried out under the weak horizontal assumptions of
Proposition~\ref{prop:functional-taylor-expansion} with only $p$ vertical
derivatives available.

Under the stronger horizontal differentiability assumptions used
by Bielert~\cite{Bielert2026}, the same  approximation
philosophy can be iterated to sufficiently high order to
yield a rough functional It\^o formula for arbitrary
$\gamma\in(0,1)$.}
\end{remark}
\subsection{Recovery of  first-order controlled regularity}
\label{sec:quadratic-variation-recovery}

The loss of regularity generated by the approximation
argument is already present in the first-order estimate of \cite{ananova2017} and persists in its higher-order extension \cite{Bielert2026}. It is therefore natural to ask whether this loss is
an artefact of the approximation procedure or reflects a
genuine feature of path-dependent functionals.

The following result shows that the loss is not intrinsic.
When the control is $\gamma-$H\"older with $\gamma>\sqrt{2}-1$, 
additional cancellations recover the classical $2\gamma$ remainder.
\begin{proposition}[Recovery of the classical first-order remainder]
\label{prop:classical-remainder}
Let
\[
\sqrt{2}-1<\gamma\le \frac12,
\qquad
X\in C^\gamma([0,T],\mathbb R^d),
\]
and assume that $F$ satisfies the hypotheses of
Proposition~\ref{prop:functional-taylor-expansion} with $p=2$.
Then, for every $0\le s<t\le T$,
\[
F(t,X_t)-F(s,X_s)
=
\nabla_\omega F(s,X_s)[X_{s,t}]
+
R^F_{s,t},\quad{\rm where}\quad
\|R^F_{s,t}\|
\le
C_{F,T,X}|t-s|^{2\gamma}.
\]
Moreover,
\[
\|\nabla_\omega F(t,X_t)-\nabla_\omega F(s,X_s)\|
\le
C_{F,T,X}|t-s|^\gamma.
\]
Consequently, $
\bigl(F(\cdot,X),\nabla_\omega F(\cdot,X)\bigr)$
is a controlled path  in the  sense of Gubinelli \cite{GUBINELLI200486}.
\end{proposition}

\begin{proof}
Set  $G:=\nabla_\omega F,
H:=\nabla_\omega^2F.$
We first construct a second-order sewn increment. Define
\[
\Xi_{s,t}
:=
G(s,X_s)[X_{s,t}]
+
\frac12
H(s,X_s)[X_{s,t}^{\otimes2}],
\qquad
0\le s\le t\le T.
\]

We claim that $\Xi$ satisfies the hypotheses of the Sewing
Lemma. Let $s<u<t$. Since
$
X_{s,t}=X_{s,u}+X_{u,t}$
and $H(s,X_s)$ is symmetric, we have
\begin{equation}
\begin{aligned}
\delta\Xi_{s,u,t}
:={}&
\Xi_{s,t}-\Xi_{s,u}-\Xi_{u,t}
\\
={}&
\bigl(G(s,X_s)-G(u,X_u)\bigr)[X_{u,t}]
+
H(s,X_s)[X_{s,u},X_{u,t}]
\\
&+
\frac12
\bigl(H(s,X_s)-H(u,X_u)\bigr)
[X_{u,t}^{\otimes2}]
\\
={}&
-\Bigl(
G(u,X_u)-G(s,X_s)
-H(s,X_s)[X_{s,u}]
\Bigr)[X_{u,t}]
\\
&+
\frac12
\bigl(H(s,X_s)-H(u,X_u)\bigr)
[X_{u,t}^{\otimes2}].
\end{aligned}
\label{eq.27}
\end{equation}

Apply Proposition~\ref{prop:functional-taylor-expansion} at level
$\ell=1$. Since $p=2$, one has $
q_1=1+\gamma,
\beta_1=\gamma(1+\gamma),$
and therefore, locally in $(s,u)$,
\begin{equation}
\left\|
G(u,X_u)-G(s,X_s)-H(s,X_s)[X_{s,u}]
\right\|
\le
C_{F,T,X}
\left(
|u-s|+|u-s|^{\gamma+\gamma^2}
\right).
\label{eq.28}
\end{equation}
Using the $\gamma$-H\"older continuity of $X$, the first term
on the right-hand side of \eqref{eq.27} is thus bounded by
\[
C_{F,T,X}
\left(
|u-s|+|u-s|^{\gamma+\gamma^2}
\right)
|t-u|^\gamma,
\]
and hence
\begin{equation}
\begin{aligned}
&
\left\|
\Bigl(
G(u,X_u)-G(s,X_s)
-H(s,X_s)[X_{s,u}]
\Bigr)[X_{u,t}]
\right\|
\le
C_{F,T,X}
\left(
|t-s|^{1+\gamma}
+
|t-s|^{2\gamma+\gamma^2}
\right).
\end{aligned}
\label{eq.29}
\end{equation}

On the other hand, horizontal and supremum-norm Lipschitz
continuity of $H$ imply
\begin{equation}
\begin{aligned}
\|H(u,X_u)-H(s,X_s)\|
&\le
\|H(u,X_u)-H(u,X_s)\|
+
\|H(u,X_s)-H(s,X_s)\|
\\
&\le
L_2^\infty
\|X_u-X_s\|_\infty
+
L_2^h|u-s|
\\
&\le
C_{F,T,X}
\left(
|u-s|^\gamma+|u-s|
\right).
\end{aligned}
\label{eq.30}
\end{equation}
Consequently,
\begin{equation}
\begin{aligned}
&
\left\|
\bigl(H(s,X_s)-H(u,X_u)\bigr)
[X_{u,t}^{\otimes2}]
\right\|
\le
C_{F,T,X}
\left(
|t-s|^{3\gamma}
+
|t-s|^{1+2\gamma}
\right).
\end{aligned}
\label{eq.31}
\end{equation}
Combining \eqref{eq.27}--\eqref{eq.31},  we obtain
\[
\|\delta\Xi_{s,u,t}\|
\le
C_{F,T,X}
\left(
|t-s|^{1+\gamma}
+
|t-s|^{2\gamma+\gamma^2}
+
|t-s|^{3\gamma}
+
|t-s|^{1+2\gamma}
\right).
\]
Since $0<\gamma\le1/2$,
\[
2\gamma+\gamma^2
\le
1+\gamma,
\qquad
2\gamma+\gamma^2
\le
3\gamma,
\qquad
2\gamma+\gamma^2
\le
1+2\gamma.
\]
Thus, locally in $(s,t)$,
\[
\|\delta\Xi_{s,u,t}\|
\le
C_{F,T,X}|t-s|^\theta,
\qquad
\theta:=2\gamma+\gamma^2.
\label{eq.32}
\]
The assumption
$\gamma>\sqrt2-1 $
is equivalent to
$
2\gamma+\gamma^2>1.$
Hence the Sewing Lemma \cite[Lemma 4.2]{frizhairer} (see Section \ref{sec:functional-calculus}) yields a unique additive map
$
I:\Delta_T\longrightarrow\mathbb R^v
$
such that
\begin{equation}
\left\|
I_{s,t}
-
G(s,X_s)[X_{s,t}]
-
\frac12H(s,X_s)[X_{s,t}^{\otimes2}]
\right\|
\le
C_{F,T,X}|t-s|^\theta.
\label{eq.33}
\end{equation}
We next relate this sewn increment to increments of the functional.
For this purpose, let $Y\in C([0,T],\mathbb R^d)$ such that
$Y|_{[s,t]}$ has bounded variation.
Define
\[
\Xi^Y_{s,t}
:=
G(s,Y_s)[Y_{s,t}]
+
\frac12H(s,Y_s)[Y_{s,t}^{\otimes2}].
\]
and
\[
I^Y_{s,t}
:=
\lim_{|\pi|\to0}
\sum_{[u,v]\in\pi}
\left(
G(u,Y_u)[Y_{u,v}]
+\frac12H(u,Y_u)[Y_{u,v}^{\otimes2}]
\right).
\]
Since $H$ is bounded on the compact family of stopped
paths generated by $Y$,
\[
\begin{aligned}
\left\|
\sum_{[u,v]\in\pi}
H(u,Y_u)[Y_{u,v}^{\otimes2}]
\right\|
&\le
C_{F,T,Y}
\sum_{[u,v]\in\pi}|Y_{u,v}|^2
\\
&\le
C_{F,T,Y}
\max_{[u,v]\in\pi}|Y_{u,v}|
\operatorname{Var}(Y;[s,t]),
\end{aligned}
\]
which converges to zero as $|\pi|\to0$. Thus the limit is equal to
\begin{equation}\label{eq.34}
I^Y_{s,t}=\int_s^tG(r,Y_r)\,dY(r).
\end{equation}
If $Y$ is in addition $\gamma$-H\"older, this limit
coincides by uniqueness with the sewn additive map associated
with $\Xi^Y$.
We shall  use the following estimate:
\begin{equation}
\left\|
F(t,Y_t)-F(s,Y_s)
-
\int_s^tG(r,Y_r)\,dY(r)
\right\|
\le
L_0^h|t-s|.
\label{eq.35}
\end{equation}
Let $Y^\pi$ be the piecewise-constant approximation of $Y$ along $\pi=\{s=t_0<\cdots<t_m=t\}$
which agrees with $Y$ at the points of $\pi$ and is constant
on each $[t_i,t_{i+1})$. Then
\begin{equation}
\begin{aligned}
&F(t,Y^\pi_t)-F(s,Y_s)
=
\sum_{i=0}^{m-1}
\Bigl(
F(t_{i+1},Y^\pi_{t_i})
-
F(t_i,Y^\pi_{t_i})
\Bigr)
+
\sum_{i=0}^{m-1}
\Bigl(
F(t_{i+1},Y^\pi_{t_{i+1}})
-
F(t_{i+1},Y^\pi_{t_i})
\Bigr).
\end{aligned}
\label{eq.36}
\end{equation}
The first sum consists only of horizontal increments and
therefore satisfies
\begin{equation}
\left\|
\sum_{i=0}^{m-1}
\Bigl(
F(t_{i+1},Y^\pi_{t_i})
-
F(t_i,Y^\pi_{t_i})
\Bigr)
\right\|
\le
L_0^h|t-s|.
\label{eq.37}
\end{equation}

For the second sum, set
\[
\Delta_iY:=Y(t_{i+1})-Y(t_i).
\]
At  $t_{i+1}$ the stopped paths
$Y^\pi_{t_i}$ and $Y^\pi_{t_{i+1}}$ differ by the  
vertical perturbation
$
\Delta_iY\,\mathbf 1_{[t_{i+1},T]}.$
Hence, by the  Lipschitz continuity of $G$,

\begin{equation}
F(t_{i+1},Y^\pi_{t_{i+1}})
-
F(t_{i+1},Y^\pi_{t_i})
=
G(t_{i+1},Y^\pi_{t_i})[\Delta_iY]
+
r_i^\pi,\qquad
\|r_i^\pi\|
\le
C|\Delta_iY|^2.
\label{eq.39}
\end{equation}
Therefore
\begin{equation}
\sum_i\|r_i^\pi\|
\le
C
\max_i|\Delta_iY|
\operatorname{Var}(Y;[s,t])
\longrightarrow0.
\label{eq.40}
\end{equation}
Moreover,
$
\|Y^\pi-Y\|_\infty\longrightarrow0,$
and the horizontal and supremum-norm Lipschitz assumptions imply
that
$
r\longmapsto G(r,Y_r)$
is continuous. Hence
\begin{equation}
\sum_{i=0}^{m-1}
G(t_{i+1},Y^\pi_{t_i})[\Delta_iY]
\longrightarrow
\int_s^tG(r,Y_r)\,dY(r).
\label{eq.41}
\end{equation}
Finally,
\[
F(t,Y^\pi_t)\mathop{\longrightarrow}^{n\to\infty} F(t,Y_t)
\]
by supremum-norm Lipschitz continuity of $F$. Passing to the
limit in \eqref{eq.36} and using \eqref{eq.37}--\eqref{eq.41} proves
\eqref{eq.35}.
Choose now a sequence of
piecewise-linear paths $X^N$ such that
\[
X^N(r)=X(r),\qquad r\le s,\qquad
X^N(s)=X(s),
\qquad
X^N(t)=X(t),
\]
\begin{equation}
\|X^N-X\|_{\infty;[s,t]}\longrightarrow0,
\qquad
{\rm and}\qquad 
\sup_N
\|X^N\|_{\gamma;[s,t]}
\le
C_\gamma\|X\|_{\gamma;[s,t]}.
\label{eq.42}
\end{equation}
Each $X^N$ has bounded variation on $[s,t]$.
Applying \eqref{eq.35} to $X^N$ and using \eqref{eq.34}, we obtain
\begin{equation}
\left\|
F(t,X^N_t)-F(s,X_s)-I^{X^N}_{s,t}
\right\|
\le
L_0^h|t-s|.
\label{eq.43}
\end{equation}
Because of the uniform bound in \eqref{eq.42}, the constants in
the sewing estimate \eqref{eq.33} may be chosen uniformly in $N$.
Furthermore,
\[
X^N_s=X_s,
\qquad
X^N_{s,t}=X_{s,t}.
\]
Hence
\begin{equation}
\begin{aligned}
&
\left\|
I^{X^N}_{s,t}
-
G(s,X_s)[X_{s,t}]
-
\frac12
H(s,X_s)[X_{s,t}^{\otimes2}]
\right\|
\le
C_{F,T,X}|t-s|^\theta.
\end{aligned}
\label{eq.44}
\end{equation}
Combining \eqref{eq.43} and \eqref{eq.44} yields
\begin{equation}
\begin{aligned}
&
\left\|
F(t,X^N_t)-F(s,X_s)
-
\nabla_{\omega}F(s,X_s)[X_{s,t}]
-
\frac12
\nabla^2_{\omega}F(s,X_s)[X_{s,t}^{\otimes2}]
\right\|
\le
L_0^h|t-s|
+
C_{F,T,X}|t-s|^\theta.
\end{aligned}
\label{eq.45}
\end{equation}
Since $F$ is Lipschitz with respect to the supremum norm and
$X^N\to X$ uniformly,
\[
F(t,X^N_t)\mathop{\longrightarrow}^{N\to\infty} F(t,X_t).
\]
Passing to the limit $N\to\infty$ in \eqref{eq.45}, we obtain
\[
\begin{aligned}
&
\left\|
F(t,X_t)-F(s,X_s)
-
\nabla_{\omega}F(s,X_s)[X_{s,t}]
-
\frac12
\nabla^2_{\omega}F(s,X_s)[X_{s,t}^{\otimes2}]
\right\|
\\
&\qquad\le
C_{F,T,X}
\left(
|t-s|+|t-s|^\theta
\right).
\end{aligned}
\label{eq.46}
\]

It follows that
\[
\begin{aligned}
\|R^F_{s,t}\|
&\le
\frac12
\|H(s,X_s)\|
|X_{s,t}|^2
+
C_{F,T,X}
\left(
|t-s|+|t-s|^\theta
\right)
\\
&\le
C_{F,T,X}
\left(
|t-s|^{2\gamma}
+
|t-s|
+
|t-s|^{2\gamma+\gamma^2}
\right).
\end{aligned}
\label{eq.47}
\]
Since
$
2\gamma\le1$ and $2\gamma+\gamma^2>1$
we have, on $[0,T]$,
\[
|t-s|+|t-s|^{2\gamma+\gamma^2}
\le
C_T|t-s|^{2\gamma}.
\]
Consequently,
\begin{equation}
\|R^F_{s,t}\|
\le
C_{F,T,X}|t-s|^{2\gamma}.
\label{eq.48}
\end{equation}
Finally,
\[
\begin{aligned}
\|G(t,X_t)-G(s,X_s)\|
&\le
\|G(t,X_t)-G(t,X_s)\|
+
\|G(t,X_s)-G(s,X_s)\|
\\
&\le
L_1^\infty
\|X_t-X_s\|_\infty
+
L_1^h|t-s|
\\
&\le
C_{F,T,X}
\left(
|t-s|^\gamma+|t-s|
\right)
\le
C_{F,T,X}|t-s|^\gamma.
\end{aligned}
\]
Together with \eqref{eq.48}, this is precisely the 
first-order controlled-path estimate.
\end{proof}
\begin{remark}[Relation with  regularity loss in Bielert (2026)]
\label{rem:bielert-loss}
Bielert~\cite[Corollary~2.3 and Remark~3.4]{Bielert2026}
obtains, under stronger causal differentiability assumptions,
higher-order approximation estimates containing a
$\gamma^2$-type loss, and points out that it is open whether this
loss can be avoided in general.
Proposition~\ref{prop:classical-remainder} shows that in the first-order controlled-path
problem the weaker exponent supplied directly by
Proposition~\ref{prop:functional-taylor-expansion} is not optimal when
\[
\sqrt2-1<\gamma\le\frac12.
\]
Indeed, a second-order compensated sewing argument recovers the
classical $2\gamma$ remainder for every
$X\in C^\gamma([0,T],\mathbb R^d)$, without assuming causal time
differentiability or any additional quadratic-variation or
non-degeneracy condition on the control. The restriction
$\gamma>\sqrt2-1$ enters through the sewing condition
\[
2\gamma+\gamma^2>1.
\]
Thus this threshold reflects the  second-order sewing
argument rather than an intrinsic obstruction to
controlled regularity.
\end{remark}
Notice that the threshold $\sqrt2-1$ concerns recovery of the
{classical controlled-path remainder}, whereas
Theorem~\ref{thm:weak-horizontal-rough-integral} already yields rough
integrability from the weaker Taylor estimates for
$\gamma>2/5$.

\noindent These results distinguish three regularity regimes for the range $\gamma\in(1/3,1/2]$. If
\[
\frac13<\gamma\le\frac25,
\]
the weak-horizontal Taylor estimates obtained above do not by
themselves yield a sewing exponent larger than one. If
\[
\frac25<\gamma\le\sqrt2-1,
\]
Theorem~\ref{thm:weak-horizontal-rough-integral} yields rough integrability,
but the available estimates do not recover the classical
$2\gamma$ controlled remainder. Finally, if
\[
\sqrt2-1<\gamma\le\frac12,
\]
Proposition~\ref{prop:classical-remainder} shows that
$
\bigl(F(\cdot,X),\nabla_\omega F(\cdot,X)\bigr)
$
is a controlled path in the usual sense \cite{GUBINELLI200486}. Thus the
threshold $\gamma>2/5$ concerns rough integrability under the
weak-horizontal estimates, whereas the stronger threshold
$\gamma>\sqrt2-1$ concerns recovery of the classical
controlled-path regularity.
\begin{corollary}
\label{cor:rough-integrability-classical}
Let $
\mathbf X=(1,X,\mathbb X)$
be a step-$2$ geometric $\gamma$-H\"older rough path over
$X\in C^\gamma([0,T],\mathbb R^d)$ with
$
\sqrt{2}-1<\gamma\le \frac12 $.
If
$
F:\Lambda_T^d\longrightarrow
\operatorname{Lin}(\mathbb R^d,\mathbb R^v)
$
satisfies the assumptions of Proposition~\ref{prop:classical-remainder}
with $p=2$, then
\[
\bigl(F(\cdot,X),\nabla_\omega F(\cdot,X)\bigr)
\]
is a controlled path with respect to $X$ and 
$\int_s^t F(r,X_r)\,d\mathbf X_r
$
may be  defined as a rough integral.
\end{corollary}

\section{From controlled path families to functional calculus}
\label{sec:converse}

We now turn to the converse direction, which is the main
representation result of the paper. Controlled rough path theory
is ordinarily formulated relative to a fixed control, and the
Gubinelli derivatives are then merely coefficients in an
expansion along that control. We show that the situation becomes
rigid when such expansions are specified compatibly as the
underlying control varies.

We consider a family of non-anticipative
functionals
\[
G_0,\ldots,G_p
\]
satisfying compatible higher-order controlled Taylor estimates.
Under a mild continuity assumption and an $o(r)$ condition on
the remainder in the control increment, we prove that
\[
G_{j}=\nabla^j_\omega G_0,
\qquad \j=1,\ldots,p.
\]
Thus the entire coefficient hierarchy is generated by the
vertical derivatives of the base functional $G_0$. In particular, the coefficients are
symmetric and uniquely determined on vertical fibres.
\subsection{Functional representation of compatible controlled families}
\label{sec:abstract-identification}

We now turn to the converse problem. The particular power
$q_\ell$ appearing in Proposition~\ref{prop:functional-taylor-expansion}
is not relevant for the identification of the coefficients.
The argument only requires that the error in a vertical
increment be of smaller order than the size of that increment.
This leads naturally to the following formulation.

\begin{theorem}[Representation theorem]
\label{thm:representation}
Let $p\in\mathbb N$ and $\gamma\in(0,1)$. For
$\ell=0,\ldots,p$, set
\[
E_\ell
:=
\operatorname{Lin}
\bigl(
(\mathbb R^d)^{\otimes\ell},
\mathbb R^k
\bigr),
\qquad
E_0:=\mathbb R^k.
\]

Let
\[
G_j:\Lambda_T^d\longrightarrow E_j,
\qquad
j=0,\ldots,p,
\]
be strongly left-continuous non-anticipative functionals.

For each $\ell=0,\ldots,p-1$, let
$
\omega_\ell:[0,1]\longrightarrow[0,\infty)
$
be non-decreasing, with
\[
\omega_\ell(0)=0,
\qquad
\lim_{r\downarrow0}
\frac{\omega_\ell(r)}{r}=0.
\]

Assume that for every continuous path
$Y\in C([0,T],\mathbb R^d)$, every
$0\le s<t\le T$ such that
$Y|_{[s,t]}\in C^\gamma$, every
$\ell=0,\ldots,p-1$, and every compact
$K\subset\mathbb R^d$, there exists
$C_{\ell,T,K}>0$ such that, whenever
\[
Y([0,t])\subset K,
\qquad
\rho_{s,t}(Y)
:=
\|Y\|_{\gamma;[s,t]}|t-s|^\gamma
\le1,
\]
one has
\begin{align}
\Bigg\|
&G_\ell(t,Y_t)-G_\ell(s,Y_s)
-
\sum_{j=1}^{p-\ell}
\frac1{j!}
G_{\ell+j}(s,Y_s)
[Y_{s,t}^{\otimes j}]
\Bigg\|_{E_\ell}
\le
C_{\ell,T,K}
\left(
|t-s|+
\omega_\ell(\rho_{s,t}(Y))
\right).
\label{eq:abstract-modulus-expansion}
\end{align}
Then, for every continuous stopped path
$(t,X)\in W_T^d$ with $t>0$, every
$a\in\mathbb R^d$, and every
$\ell=0,\ldots,p-1$, the map
$
h\in \mathbb R^d\mapsto G_\ell(t,X_t^{a+h})\in 
 E_\ell,
$
is differentiable at $h=0$, and
\[
D_h
G_\ell(t,X_t^{a+h})\big|_{h=0}
=
\iota_\ell
\bigl(G_{\ell+1}(t,X_t^a)\bigr),
\]
where
$
\iota_\ell:
E_{\ell+1}
\longrightarrow
\operatorname{Lin}(\mathbb R^d,E_\ell)
$ is given by
\[
[\iota_\ell(A)h](v_1,\ldots,v_\ell)
:=
A(v_1,\ldots,v_\ell,h).
\]
Consequently,
\[
\nabla_\omega G_\ell(t,X_t^a)
=
G_{\ell+1}(t,X_t^a),
\qquad
\ell=0,\ldots,p-1,
\]
and the map $
a\longmapsto G_0(t,X_t^a)$
is $p$ times differentiable with
\[
\partial^j_{a}
 G_0(t,X_t^a)(a)
=
G_j(t,X_t^a),
\qquad
j=1,\ldots,p.
\]
In particular,
\[
\nabla_\omega^jG_0(t,X)=G_j(t,X),
\qquad
j=1,\ldots,p.
\]
If the corresponding identities are assumed at $t=0$ for
the constant paths $
\overline{x+a},$
then the same conclusions hold at $t=0$.
\end{theorem}
\begin{proof}
Fix $\ell$, $(t,X)\in W_T^d$ with $t>0$, and
$a\in\mathbb R^d$. Let $|h|\le1$ and
$\varepsilon\in(0,t/2)$, and use the same two-ramp path
\[
z^{\varepsilon,a,h}(r)
=
\begin{cases}
X(r),
&0\le r<t-2\varepsilon,
\\[1mm]
X(t-2\varepsilon)
+\dfrac{r-(t-2\varepsilon)}{\varepsilon}a,
&t-2\varepsilon\le r<t-\varepsilon,
\\[2mm]
X(t-2\varepsilon)+a+
\left(
\dfrac{r-(t-\varepsilon)}{\varepsilon}
\right)^\gamma h,
&t-\varepsilon\le r\le t,
\\[2mm]
X(t-2\varepsilon)+a+h,
&t<r\le T.
\end{cases}
\]
On $[t-\varepsilon,t]$,
\[
\|z^{\varepsilon,a,h}\|_{\gamma;[t-\varepsilon,t]}
\varepsilon^\gamma
\le|h|.
\]
For fixed $a$ and $|h|\le1$, the ranges of these paths are
contained in one compact set $K$ for all sufficiently small
$\varepsilon$. Hence
\eqref{eq:abstract-modulus-expansion} gives
\begin{align*}
\Bigg\|
&G_\ell(t,z_t^{\varepsilon,a,h})
-G_\ell(t-\varepsilon,z_{t-\varepsilon}^{\varepsilon,a,h})
-
\sum_{j=1}^{p-\ell}
\frac1{j!}
G_{\ell+j}
(t-\varepsilon,z_{t-\varepsilon}^{\varepsilon,a,h})
[h^{\otimes j}]
\Bigg\|
\le
C_{\ell,T,K}
\left(
\varepsilon+\omega_\ell(|h|)
\right).
\end{align*}
Strong left continuity and $\varepsilon\downarrow0$ yield
\begin{align}
\Bigg\|
&G_\ell(t,X_t^{a+h})-G_\ell(t,X_t^a)
-
\sum_{j=1}^{p-\ell}
\frac1{j!}
G_{\ell+j}(t,X_t^a)[h^{\otimes j}]
\Bigg\|
\le
C_{\ell,T,K}\omega_\ell(|h|).
\label{eq:fibre-Taylor-modulus}
\end{align}

Isolating the linear term gives
\begin{align*}
&
\frac{
\|G_\ell(t,X_t^{a+h})
-G_\ell(t,X_t^a)
-G_{\ell+1}(t,X_t^a)[h]\|
}{|h|}
\le
C_{\ell,T,K}
\frac{\omega_\ell(|h|)}{|h|}
+
\sum_{j=2}^{p-\ell}
\frac{
\|G_{\ell+j}(t,X_t^a)\|
}{j!}
|h|^{j-1}.
\end{align*}
The right-hand side converges to zero as $h\to0$. Hence
\[
DG_\ell(t,X_t^a)
=
\iota_\ell(G_{\ell+1}(t,X_t^a)).
\]

Now set
$
g_\ell(a):=G_\ell(t,X_t^a).$
We have
$
Dg_\ell(a)=\iota_\ell(g_{\ell+1}(a)),
\ell=0,\ldots,p-1.
$
Iterating gives
$
D^jg_0(a)=G_j(t,X_t^a),
j=1,\ldots,p,
$
with the natural multilinear identification. This proves the
claim.
\end{proof}
\begin{corollary}
The conclusion of
Theorem~\ref{thm:representation}
holds in particular when
$\omega_\ell(r)=r^{q_\ell},
q_\ell>1.$
Thus the previous power-type formulation is a special case of
the representation theorem.
\end{corollary}
\subsection{Rigidity, symmetry and uniqueness}\label{sec:rigidity}
\begin{corollary}[Symmetry]
\label{cor:symmetry}
Under the assumptions of
Theorem~\ref{thm:representation}, for every
$(t,X)\in W_T^d$ with $t>0$, every $a\in\mathbb R^d$, and
$j=1,\ldots,p$,
\[
G_j(t,X_t^a)
\in
\operatorname{Lin}_{\mathrm{sym}}
\bigl(
(\mathbb R^d)^{\otimes j},
\mathbb R^k
\bigr).
\]
\end{corollary}

\begin{proof}
By Theorem~\ref{thm:representation},
\[
G_j(t,X_t^a)
=
D^j
\bigl[
b\mapsto G_0(t,X_t^b)
\bigr](a).
\]
Higher Fr\'echet derivatives are symmetric multilinear maps.
\end{proof}
\begin{corollary}[Uniqueness of compatible coefficient hierarchies]
\label{cor:uniqueness}
Let
\[
(G_0,G_1,\ldots,G_p)
\qquad\text{and}\qquad
(G_0,\widetilde G_1,\ldots,\widetilde G_p)
\]
be two strongly left-continuous families satisfying the
assumptions of
Theorem~\ref{thm:representation}, possibly with
different moduli $\omega_\ell$.
Then, for every continuous stopped path
$(t,X)\in W_T^d$ with $t>0$, every
$a\in\mathbb R^d$, and every $j=1,\ldots,p$,
\[
G_j(t,X_t^a)
=
\widetilde G_j(t,X_t^a).
\]
\end{corollary}

\begin{proof}
Both sides are equal to $
\partial^j_b
\bigl[
b\mapsto G_0(t,X_t^b)
\bigr](a).$
\end{proof}
The preceding results show that the relation between
functional calculus and controlled paths has two complementary
directions.

Under weak horizontal regularity, the  vertical
derivatives generate a higher-order controlled hierarchy,
although in general with weaker, level-dependent remainder
regularity. In the first-order regime, a second-order 
sewing argument recovers the classical controlled-path
regularity for $\gamma>\sqrt2-1$.

Conversely, the coefficient hierarchy is rigid once it is
required to arise from non-anticipative functionals that are
compatible across controls: any strongly left-continuous
hierarchy satisfying an $o(r)$ controlled Taylor estimate is
necessarily the vertical-derivative hierarchy of its
zeroth-order component. In particular, its coefficients are
symmetric and uniquely determined.
\section{Functional transformations of compatible controlled families}
\label{sec.chainrule}

It is well known that controlled paths are stable under composition with smooth functions \cite[Sec. 7.4]{frizhairer}, and that
the Gubinelli coefficients of their local expansions may be computed using
the usual chain rule \cite[Sec. 7.3 \& 7.6]{frizhairer}.

A first-order functional analogue of this result was established in
\cite[Theorem~3.17]{ananova2023}: if $(Y,Y')$ is controlled by a rough path
$X$ and $F$ is a regular non-anticipative functional, then
\[
\bigl(F(\cdot,Y),\nabla_\omega F(\cdot,Y)Y'\bigr)
\]
is again controlled by $X$.

We establish in this section the corresponding higher-order
statement for compatible controlled path families. The main result is a
stability theorem (Theorem \ref{thm:functional-chain-rule}): if two compatible coefficient families are composed
through an admissible non-anticipative transformation, then the coefficients
obtained from the higher-order chain rule again form a compatible coefficient
family. Consequently, along each H\"older control they define an order-$p$
controlled expansion, and Theorem~\ref{thm:representation} identifies the
transformed coefficients with the vertical derivatives of the transformed
zeroth-order functional.

\subsection{Chain rule and stability of compatible families}

\begin{definition}[Admissible causal transformation]
\label{def:admissible-causal}
Let
$
G_0:\Lambda_T^d\longrightarrow\mathbb R^m
$
be non-anticipative and define
$
\Gamma_{G_0}(X)(t):=G_0(t,X_t).
$
We call $\Gamma_{G_0}$ admissible if:
\begin{enumerate}
\item $\Gamma_{G_0}$ maps c\`adl\`ag paths to c\`adl\`ag paths
      and continuous paths to continuous paths;
\item whenever
      $(t_n,X^n_{t_n})$ converges strongly from the left to
      $(t,X_t)$ in the sense of
      Definition~\ref{def:strong-left-continuity}, the stopped
      paths
      $
      (t_n,\Gamma_{G_0}(X^n)_{t_n})
      $
      converge strongly from the left to
      $
      (t,\Gamma_{G_0}(X)_t).
      $
\end{enumerate}
\end{definition}
Sufficient conditions for admissibility of a map are given e.g. in \cite[Lemma 2.5, Thm 2.6]{CF10B}.
\begin{theorem}[Chain rule and stability under functional transformations]
\label{thm:functional-chain-rule}
\label{thm:functional-stability}
Let $p\in\mathbb N$ and $\gamma\in(0,1)$ and
$
(G_0,\ldots,G_p),(F_0,\ldots,F_p)$
be families of non-anticipative functionals
\[
G_j:
\Lambda_T^d
\longrightarrow
\operatorname{Lin}_{\mathrm{sym}}
\bigl((\mathbb R^d)^{\otimes j},\mathbb R^m\bigr),\qquad
F_j:
\Lambda_T^m
\longrightarrow
\operatorname{Lin}_{\mathrm{sym}}
\bigl((\mathbb R^m)^{\otimes j},\mathbb R^n\bigr).
\]
which are
boundedness-preserving and satisfy the assumptions of
Theorem~\ref{thm:representation}, with remainder moduli
$\omega_j^G$ and $\omega_j^F$, respectively, and assume that
$\Gamma_{G_0}$ is admissible in the sense of
Definition~\ref{def:admissible-causal}.
For $X\in D([0,T],\mathbb R^d)$ set
\[
Y^X:=\Gamma_{G_0}(X),
\qquad
Y^X(t)=G_0(t,X_t).
\]
Then $(K_0,\ldots,K_p)$ defined by \begin{equation}
K_0(t,X_t):=F_0(t,Y^X_t),\quad K_\ell(t,X_t)[v_1,\ldots,v_\ell]
:=
\sum_{\pi\in\mathcal P_\ell}
F_{|\pi|}(t,Y^X_t)
\left[
G_{|B|}(t,X_t)[v_B]
\right]_{B\in\pi}
\label{eq:17}
\end{equation}
\begingroup
\makeatletter
\edef\@currentlabel{\theequation}
\label{eq:*}
\makeatother
\endgroup
where $\mathcal P_\ell$ is the set of partitions of
$\{1,\ldots,\ell\}$ and for
$B=\{i_1,\ldots,i_r\}$,
\[
G_{|B|}(t,X_t)[v_B]
:=
G_r(t,X_t)[v_{i_1},\ldots,v_{i_r}]
\]
is again a compatible coefficient family in the sense of
Theorem~\ref{thm:representation}: there exists a non-decreasing modulus
\begin{equation}
\omega_\ell^K:[0,1]\longrightarrow[0,\infty),
\qquad
\omega_\ell^K(0)=0,
\qquad
\frac{\omega_\ell^K(r)}r\mathop{\longrightarrow}^{r\downarrow0}0,
\label{eq:18}
\end{equation}
such that, whenever $X\in C^\gamma([s,t])$,
its range up to $t$ is contained in a fixed compact set, and
\[
\rho_{s,t}(X)
:=
\|X\|_{\gamma;[s,t]}|t-s|^\gamma
\le1
\] we have
\begin{equation}
\begin{aligned}
&
\left\|
K_\ell(t,X_t)-K_\ell(s,X_s)
-
\sum_{j=1}^{p-\ell}
\frac1{j!}
K_{\ell+j}(s,X_s)[X_{s,t}^{\otimes j}]
\right\|\le
C_{\ell,T,K}
\left(
|t-s|+
\omega_\ell^K(\rho_{s,t}(X))
\right)
\end{aligned}
\label{eq:19}
\end{equation}
Thus, for every such control $X$, the  hierarchy
$
t\longmapsto
\bigl(K_0(t,X_t),\ldots,K_p(t,X_t)\bigr)$
defines an order-$p$ controlled path hierarchy along $X$,  compatible in the
sense of \eqref{eq:19}. Moreover
\[
K_\ell=\nabla_\omega^\ell K_0,
\qquad
\ell=1,\ldots,p
\]
on the set of continuous paths.
\end{theorem}
Formula  \eqref{eq:17} is precisely the  functional chain rule
for the transformed functional
\[
K_0(t,X_t)=F_0\bigl(t,\Gamma_{G_0}(X)_t\bigr).
\]
\subsection{Finite-jet calculus and proof of Theorem \ref{thm:functional-chain-rule}}
Functional $p$-jets provide a natural setting for transformations of
controlled expansions \cite{cont2019}. Let $V,W$ be finite-dimensional
vector spaces and set
\[
\mathcal J^p(V,W)
:=
W\oplus
\bigoplus_{j=1}^p
\operatorname{Lin}_{\mathrm{sym}}(V^{\otimes j},W).
\]
For $
A=(A_0,\ldots,A_p)\in\mathcal J^p(V,W)
$
and $h\in V$, define the translated jet
\[
\tau_hA
:=
\bigl((\tau_hA)_0,\ldots,(\tau_hA)_p\bigr)
\]
\begin{equation}
{\rm by}\qquad (\tau_hA)_\ell
:=
\sum_{j=0}^{p-\ell}
\frac1{j!}
A_{\ell+j}[h^{\otimes j}],
\qquad
\ell=0,\ldots,p.
\label{eq:13}
\end{equation}
Here, as before, contraction is understood in the last $j$ variables.
If
\[
A\in\mathcal J^p(V,W),
\qquad
B\in\mathcal J^p(W,U),
\]
we denote by
\[
B\circ_p A\in\mathcal J^p(V,U)
\]
their truncated jet composition, defined by
\begin{equation}
\begin{aligned}
(B\circ_p A)_0&:=B_0,\qquad
(B\circ_p A)_\ell[v_1,\ldots,v_\ell]
:=
\sum_{\pi\in\mathcal P_\ell}
B_{|\pi|}
\left[
A_{|B|}[v_B]
\right]_{B\in\pi},
\qquad \ell\ge1.
\end{aligned}
\label{eq:14}
\end{equation}
Here $\mathcal P_\ell$ is the set of partitions of
$\{1,\ldots,\ell\}$. Formula \eqref{eq:14} is the usual higher-order
chain rule written at the level of finite jets.
We shall use the following elementary consequence:
\begin{lemma}[Translation and composition of truncated jets]
\label{lem:jet-composition}
Let $A\in\mathcal J^p(V,W)$ and
$B\in\mathcal J^p(W,U)$ belong to a bounded set. For $h\in V$ set
\[
\widehat y
:=
(\tau_hA)_0-A_0
=
\sum_{j=1}^p\frac1{j!}A_j[h^{\otimes j}].
\]
Then, for $\ell=0,\ldots,p-1$ and $|h|\le1$,
\begin{equation}
\left\|
\bigl((\tau_{\widehat y}B)\circ_p(\tau_hA)\bigr)_\ell
-
\bigl(\tau_h(B\circ_pA)\bigr)_\ell
\right\|
\le
C|h|^{p+1-\ell},
\label{eq:15}
\end{equation}
where $C$ is uniform on bounded subsets of the two jet spaces.
The maps
$
(A,B)\longmapsto B\circ_pA
$
and
$
(y,B)\longmapsto\tau_yB
$
are componentwise locally Lipschitz on bounded sets.
\end{lemma}

\begin{proof}
Associate with $A$ and $B$ their formal Taylor polynomials
\[
P_A(z)
=
A_0+\sum_{j=1}^p\frac1{j!}A_j[z^{\otimes j}],
\qquad
P_B(w)
=
B_0+\sum_{j=1}^p\frac1{j!}B_j[w^{\otimes j}].
\]
Set
\[
Q(z):=P_B\bigl(P_A(z)-A_0\bigr),
\]
and let $Q_{\le p}$ denote the truncation of $Q$ to total degree at most
$p$. By the higher-order chain rule, the $p$-jet of $Q$ at the origin is
$B\circ_pA$, hence
\[
\bigl(\tau_h(B\circ_pA)\bigr)_\ell
=
D^\ell Q_{\le p}(h).
\]
On the other hand, $\tau_hA$ is the $p$-jet of $P_A$ at $h$, while
$\tau_{\widehat y}B$ is the $p$-jet of $P_B$ at
$\widehat y=P_A(h)-A_0$. Applying the higher-order chain rule at $h$
gives
\[
\bigl((\tau_{\widehat y}B)\circ_p(\tau_hA)\bigr)_\ell
=
D^\ell Q(h).
\]
Therefore
\[
\bigl((\tau_{\widehat y}B)\circ_p(\tau_hA)\bigr)_\ell
-
\bigl(\tau_h(B\circ_pA)\bigr)_\ell
=
D^\ell(Q-Q_{\le p})(h).
\]
Since $Q-Q_{\le p}$ contains only monomials of total degree at least
$p+1$, its $\ell$th derivative is bounded by
$C|h|^{p+1-\ell}$ on bounded sets. This proves \eqref{eq:15}.

The local Lipschitz assertions follow immediately from the fact that both
jet composition and jet translation are polynomial maps between
finite-dimensional spaces.
\end{proof}

\begin{proof}[Proof of Theorem~\ref{thm:functional-chain-rule}]
Fix a compact set $K\subset\mathbb R^d$ and a continuous path $X$ whose
range up to time $t$ is contained in $K$. Write
\[
x:=X_{s,t},
\qquad
\rho:=\rho_{s,t}(X).
\]

For $r\in[0,T]$, introduce the jets
$
\mathbf G_r
:=
\bigl(G_0(r,X_r),\ldots,G_p(r,X_r)\bigr)
$
and
\[
\mathbf F_r
:=
\bigl(F_0(r,Y^X_r),\ldots,F_p(r,Y^X_r)\bigr).
\]
By \eqref{eq:17},
\begin{equation}
\mathbf K_r
=
\mathbf F_r\circ_p\mathbf G_r.
\label{eq:20}
\end{equation}

We first estimate the regularity of the transformed control $Y^X$. Let
$s\le u<v\le t$ and put
\[
\rho_{u,v}
:=
\|X\|_{\gamma;[s,t]}|v-u|^\gamma.
\]
Since $\rho_{u,v}\le\rho\le1$, the level-zero compatibility estimate for
$G$ gives
\begin{equation}
Y^X_{u,v}
=
\sum_{j=1}^p
\frac1{j!}
G_j(u,X_u)[X_{u,v}^{\otimes j}]
+
E^{G,0}_{u,v},
\label{eq:21}
\end{equation}
where
\begin{equation}
\|E^{G,0}_{u,v}\|
\le
C\left(
|v-u|+\omega_0^G(\rho_{u,v})
\right).
\label{eq:22}
\end{equation}
Since $\omega_0^G(r)=o(r)$, there exists $C_G$ such that
\begin{equation}
\omega_0^G(r)\le C_Gr,
\qquad 0\le r\le1.
\label{eq:23}
\end{equation}
The coefficient functionals $G_j$ are bounded on the compact family of
stopped paths under consideration. Hence
\[
|Y^X_{u,v}|
\le
C\left(|v-u|+\rho_{u,v}\right).
\]
Using
\[
|v-u|
\le
|t-s|^{1-\gamma}|v-u|^\gamma,
\]
we obtain
\[
\|Y^X\|_{\gamma;[s,t]}
\le
C\left(
\|X\|_{\gamma;[s,t]}
+
|t-s|^{1-\gamma}
\right).
\]
Consequently,
\begin{equation}
\rho_{s,t}(Y^X)
\le
C\left(\rho+|t-s|\right).
\label{eq:24}
\end{equation}

We first argue locally, with $\rho+|t-s|$ sufficiently small that the
right-hand side of \eqref{eq:24} is at most $1$. The compatibility estimate
for $F$ may then be applied along $Y^X$. The remaining range $\rho\le1$ is
handled at the end of the proof by boundedness preservation.

We now write the compatibility relations in jet form. Set
$
\widehat{\mathbf G}_t
:=
\tau_x\mathbf G_s.
$
Then, for $j=0,\ldots,p-1$,
\begin{equation}
\|(\mathbf G_t-\widehat{\mathbf G}_t)_j\|
\le
C\left(
|t-s|+\omega_j^G(\rho)
\right).
\label{eq:25}
\end{equation}
Set
\begin{equation}
\widehat y
:=
(\widehat{\mathbf G}_t)_0-(\mathbf G_s)_0
=
\sum_{j=1}^p
\frac1{j!}
G_j(s,X_s)[x^{\otimes j}].
\label{eq:26}
\end{equation}
Since
$
y:=Y^X_{s,t}
=
G_0(t,X_t)-G_0(s,X_s),
$
the level-zero estimate in \eqref{eq:25} gives
\begin{equation}
|y-\widehat y|
\le
C\left(
|t-s|+\omega_0^G(\rho)
\right).
\label{eq:27}
\end{equation}

Similarly, applying the compatibility estimate for $F$ along the path
$Y^X$ gives, for $j=0,\ldots,p-1$,
\begin{equation}
\left\|
(\mathbf F_t)_j
-
(\tau_y\mathbf F_s)_j
\right\|
\le
C\left(
|t-s|
+
\omega_j^F(\rho_{s,t}(Y^X))
\right).
\label{eq:28}
\end{equation}

Fix now $\ell\in\{0,\ldots,p-1\}$. By \eqref{eq:20},
$
(\mathbf K_t)_\ell
=
(\mathbf F_t\circ_p\mathbf G_t)_\ell.
$
On the other hand,
\[
(\tau_x\mathbf K_s)_\ell
=
\left(
\tau_x(\mathbf F_s\circ_p\mathbf G_s)
\right)_\ell.
\]
Insert the intermediate jets
$
(\tau_y\mathbf F_s)\circ_p\mathbf G_t
$
and
$
(\tau_{\widehat y}\mathbf F_s)\circ_p(\tau_x\mathbf G_s).
$
We obtain
\begin{equation}
\begin{aligned}
\|(\mathbf K_t)_\ell-(\tau_x\mathbf K_s)_\ell\|
&\le
\left\|
(\mathbf F_t\circ_p\mathbf G_t)_\ell
-
((\tau_y\mathbf F_s)\circ_p\mathbf G_t)_\ell
\right\|
\\
&\quad+
\left\|
((\tau_y\mathbf F_s)\circ_p\mathbf G_t)_\ell
-
((\tau_{\widehat y}\mathbf F_s)
   \circ_p(\tau_x\mathbf G_s))_\ell
\right\|
\\
&\quad+
\left\|
((\tau_{\widehat y}\mathbf F_s)
   \circ_p(\tau_x\mathbf G_s))_\ell
-
(\tau_x(\mathbf F_s\circ_p\mathbf G_s))_\ell
\right\|.
\end{aligned}
\label{eq:29}
\end{equation}

All jets involved remain in a bounded set. Indeed, boundedness preservation
of the $G_j$ implies that the range of $Y^X$ up to $t$ lies in a compact
subset of $\mathbb R^m$, and boundedness preservation of the $F_j$ then
gives uniform bounds for $\mathbf F$. Hence all local Lipschitz constants in
Lemma~\ref{lem:jet-composition} may be chosen uniformly.

By \eqref{eq:28} and local Lipschitz continuity of jet composition, the first
term on the right-hand side of \eqref{eq:29} is bounded by
\begin{equation}
C\left(
|t-s|
+
\max_{0\le j\le\ell}
\omega_j^F(\rho_{s,t}(Y^X))
\right).
\label{eq:30}
\end{equation}
For the second term in \eqref{eq:29}, first change $\mathbf G_t$ into
$\tau_x\mathbf G_s$ and then replace $y$ by $\widehat y$. Using
\eqref{eq:25}, \eqref{eq:27}, and the local Lipschitz properties in
Lemma~\ref{lem:jet-composition}, we obtain
\begin{equation}
C\left(
|t-s|
+
\max_{0\le j\le\ell}\omega_j^G(\rho)
\right).
\label{eq:31}
\end{equation}

Finally, Lemma~\ref{lem:jet-composition} gives
\[
\left\|
((\tau_{\widehat y}\mathbf F_s)
   \circ_p(\tau_x\mathbf G_s))_\ell
-
(\tau_x(\mathbf F_s\circ_p\mathbf G_s))_\ell
\right\|
\le
C|x|^{p+1-\ell}.
\]
Since $|x|\le\rho$,
\begin{equation}
|x|^{p+1-\ell}
\le
\rho^{p+1-\ell}.
\label{eq:32}
\end{equation}

Combining \eqref{eq:29}--\eqref{eq:32}, we conclude that
\begin{equation}
\begin{aligned}
\|(\mathbf K_t)_\ell-(\tau_x\mathbf K_s)_\ell\|
\le C\Bigg(
&|t-s|
+
\max_{0\le j\le\ell}\omega_j^G(\rho)
+
\max_{0\le j\le\ell}
\omega_j^F(\rho_{s,t}(Y^X))
+
\rho^{p+1-\ell}
\Bigg).
\end{aligned}
\label{eq:33}
\end{equation}

It remains only to express the outer modulus in terms of $\rho$. Define
\[
\Omega_\ell^F(r)
:=
\max_{0\le j\le\ell}\omega_j^F(r),
\qquad
\Omega_\ell^G(r)
:=
\max_{0\le j\le\ell}\omega_j^G(r).
\]
Both are non-decreasing and satisfy
\begin{equation}
\Omega_\ell^F(r)=o(r),
\qquad
\Omega_\ell^G(r)=o(r).
\label{eq:34}
\end{equation}
By \eqref{eq:24},
\[
\rho_{s,t}(Y^X)
\le
C_0(\rho+|t-s|).
\]

For sufficiently small $\rho+|t-s|$, we distinguish two cases. If
$|t-s|\le\rho$, then
\[
\Omega_\ell^F(\rho_{s,t}(Y^X))
\le
\Omega_\ell^F(2C_0\rho).
\]
If instead $\rho<|t-s|$, then since $\Omega_\ell^F(r)=o(r)$ and therefore
$\Omega_\ell^F(r)\le Cr$ near zero,
\[
\Omega_\ell^F(\rho_{s,t}(Y^X))
\le
C|t-s|.
\]
Thus
\begin{equation}
\Omega_\ell^F(\rho_{s,t}(Y^X))
\le
C|t-s|
+
\widetilde\Omega_\ell^F(\rho),
\label{eq:35}
\end{equation}
where, near zero,
\begin{equation}
\widetilde\Omega_\ell^F(r)
:=
\Omega_\ell^F(2C_0r),
\qquad
\widetilde\Omega_\ell^F(r)=o(r).
\label{eq:36}
\end{equation}

Substituting \eqref{eq:35} into \eqref{eq:33}, we obtain
\begin{equation}
\|(\mathbf K_t)_\ell-(\tau_x\mathbf K_s)_\ell\|
\le
C\left(
|t-s|
+
\Omega_\ell^G(\rho)
+
\widetilde\Omega_\ell^F(\rho)
+
\rho^{p+1-\ell}
\right).
\label{eq:37}
\end{equation}
Since $\ell\le p-1$, one has $p+1-\ell\ge2$, and therefore
$r^{p+1-\ell}=o(r)$. After a non-decreasing extension away from the origin,
we may consequently define
\begin{equation}
\omega_\ell^K(r)
:=
C\left(
\Omega_\ell^G(r)
+
\widetilde\Omega_\ell^F(r)
+
r^{p+1-\ell}
\right),
\qquad 0\le r\le1.
\label{eq:38}
\end{equation}
Then
\[
\frac{\omega_\ell^K(r)}r\longrightarrow0
\qquad\text{as }r\downarrow0.
\]
Since, by definition of jet translation,
\[
(\tau_x\mathbf K_s)_\ell
=
K_\ell(s,X_s)
+
\sum_{j=1}^{p-\ell}
\frac1{j!}
K_{\ell+j}(s,X_s)[x^{\otimes j}],
\]
estimate \eqref{eq:37} is exactly \eqref{eq:19}.

The preceding argument proves the estimate locally near
$\rho+|t-s|=0$. The remaining range $\rho\le1$ is obtained by increasing
the constant and extending $\omega_\ell^K$ monotonically on $[0,1]$, using
boundedness preservation of the coefficient functionals.

It remains to verify the structural assumptions. Non-anticipativity of the
$K_\ell$ follows from non-anticipativity of the two families. Boundedness
preservation follows from \eqref{eq:17}: on compact sets the $G_j$ are
bounded, hence $Y^X$ has bounded range, and the $F_j$ are bounded on the
corresponding compact subset of $\mathbb R^m$.
Finally, assume 
$
(t_n,X^n_{t_n})\to (t,X_t)
$
 strongly from the left. By admissibility of
$\Gamma_{G_0}$,
$
(t_n,Y^{X^n}_{t_n})\to(t,Y^X_t)
$
strongly from the left. Strong left continuity of the $F_j$ and $G_j$,
together with the finite polynomial expression \eqref{eq:17}, therefore
gives
\[
K_\ell(t_n,X^n_{t_n})
\longrightarrow
K_\ell(t,X_t).
\]
Thus every $K_\ell$ is strongly left-continuous. Hence
$(K_0,\ldots,K_p)$ satisfies all the assumptions of the compatible class in
Theorem~\ref{thm:representation}.
Applying Theorem~\ref{thm:representation} to $K$ gives
\[
K_\ell=\nabla_\omega^\ell K_0,
\qquad
\ell=1,\ldots,p,
\]
on vertical fibres over continuous stopped paths. Applying the same theorem
to $F$ and $G$ identifies $F_j$ and $G_j$ with the corresponding vertical
derivatives of $F_0$ and $G_0$. Hence \eqref{eq:17} is the asserted
higher-order chain rule.
\end{proof}
Finally, applying Theorem~\ref{thm:representation} to the families $F$, $G$, and $K$ we obtain:
\begin{corollary}[Functional chain rule for Gubinelli coefficients]
\label{cor:functional-chain-rule}
Under the assumptions of Theorem~\ref{thm:functional-stability}, the
transformed family $K=(K_0,\ldots,K_p)$ is a compatible coefficient family
and
\[
K_\ell=\nabla_\omega^\ell K_0,
\qquad
\ell=1,\ldots,p,
\]
over continuous stopped paths. Consequently,
\[
\begin{aligned}
\nabla_\omega^\ell K_0(t,X_t)[v_1,\ldots,v_\ell]
=
\sum_{\pi\in\mathcal P_\ell}
&\nabla_\omega^{|\pi|}F_0(t,Y^X_t)
\left[
\nabla_\omega^{|B|}G_0(t,X_t)[v_B]
\right]_{B\in\pi}.
\end{aligned}
\]
\end{corollary}

\section{Application to path-dependent rough differential equations}
\label{sec:application-rde}
We illustrate how the representation and stability results of
Sections~\ref{sec:converse} and \ref{sec.chainrule} may be combined with the controlled-path fixed-point
method used  in
\cite{kwossek2025} to obtain new existence and uniqueness results for solutions of rough functional differential equations. 

For a path
$Y:[0,T]\to\mathbb R^m$, define 
\begin{equation}
\label{eq:volterra-memory}
\mathcal M_t(Y)
:=
\int_0^t K(t,r)\psi(Y_r)\,dr.
\end{equation}
where 
\[
K:\Delta_T\longrightarrow\operatorname{Lin}(\mathbb R^h,\mathbb R^e)\qquad \Delta_T:=\{(t,r):0\le r\le t\le T\}
\]
is a Volterra kernel.
Since
\begin{equation}
\label{eq:volterra-memory-derivative}
\frac{d}{dt}\mathcal M_t(Y)
=
K(t,t)\psi(Y_t)
+
\int_0^t
\partial_1K(t,r)\psi(Y_r)\,dr,
\end{equation}
the evolution of  $M_t(Y)$ contains a path-dependent term.
We consider the path-dependent rough differential equation
\begin{equation}
\label{eq:volterra-rde}
Y_t
=
y_0
+
\int_0^t
b\bigl(r,Y_r,\mathcal M_r(Y)\bigr)\,dr
+
\int_0^t
\sigma\bigl(r,Y_r,\mathcal M_r(Y)\bigr)\,d\mathbf X_r .
\end{equation}

\begin{proposition}[RDE with Volterra memory]
\label{prop:volterra-memory-rde}
Let $q\in(2,3)$ and let $\mathbf X$ be a continuous geometric
$q$-rough path over $\mathbb R^d$. Assume that
$
K,\partial_1K
$
are continuous and bounded on $\Delta_T$,
\[
\psi\in C_b^2(\mathbb R^m;\mathbb R^h),\quad
b\in
C_b^3\bigl(
[0,T]\times\mathbb R^m\times\mathbb R^e;
\mathbb R^m
\bigr),\quad
\sigma\in
C_b^3\bigl(
[0,T]\times\mathbb R^m\times\mathbb R^e;
\operatorname{Lin}(\mathbb R^d,\mathbb R^m)
\bigr).
\]
Then, for every $y_0\in\mathbb R^m$,
\eqref{eq:volterra-rde} admits a unique solution
$(Y,Y')$ controlled by $\mathbf X$, with
\begin{equation}
\label{eq:volterra-solution-derivative}
Y'_t
=
\sigma\bigl(t,Y_t,\mathcal M_t(Y)\bigr).
\end{equation}
Then  $(\Sigma_0,\Sigma_1,\Sigma_2)$ defined by
\begin{eqnarray}
\Sigma_0(t,\eta_t)
:=
\sigma\bigl(t,\eta(t),\mathcal M_t(\eta)\bigr),
& \Sigma_1(t,\eta_t)[v]:=
D_y\sigma\bigl(t,\eta(t),\mathcal M_t(\eta)\bigr)[v], \label{eq:volterra-vertical-derivative}\\
\Sigma_2(t,\eta_t)[v_1,v_2]
&:=
D_y^2\sigma\bigl(t,\eta(t),\mathcal M_t(\eta)\bigr)
[v_1,v_2]
\end{eqnarray}
is a compatible coefficient family of order two and
\begin{equation}
\nabla_\omega\Sigma_0=\Sigma_1,
\qquad
\nabla_\omega^2\Sigma_0=\Sigma_2
\end{equation}
 over continuous stopped paths.
Thus for every controlled path $(Y,Y')$,
$
t\longmapsto
\Sigma(t,Y_t)$
is controlled by $\mathbf X$, with Gubinelli derivative
\begin{equation}
\label{eq:volterra-coefficient-gubinelli}
\Sigma'(Y,Y')_t
=
D_y\sigma\bigl(
t,Y_t,\mathcal M_t(Y)
\bigr)Y'_t.
\end{equation}
\end{proposition}

\begin{proof}
Define
\[
G_0(t,\eta_t)
:=
\bigl(
\eta(t),\mathcal M_t(\eta)
\bigr),
\qquad
G_1(t,\eta_t)[v]:=(v,0),
\qquad
G_j:=0,\quad j\ge2.
\]
For $0\le s<t\le T$,
\begin{align}
\mathcal M_t(\eta)-\mathcal M_s(\eta)
={}&
\int_0^s
\bigl(K(t,r)-K(s,r)\bigr)\psi(\eta_r)\,dr
+
\int_s^t
K(t,r)\psi(\eta_r)\,dr.
\label{eq:volterra-memory-increment}
\end{align}
Hence, using boundedness of $K$, $\partial_1K$ and $\psi$,
\begin{equation}
\label{eq:volterra-memory-lipschitz-time}
|\mathcal M_t(\eta)-\mathcal M_s(\eta)|
\le C|t-s|.
\end{equation}
It follows that
\[
G_0(t,\eta_t)-G_0(s,\eta_s)
-G_1(s,\eta_s)[\eta_{s,t}]
=
\bigl(
0,
\mathcal M_t(\eta)-\mathcal M_s(\eta)
\bigr),
\]
while $G_1$ is constant and all higher coefficients vanish.
Thus $(G_0,G_1,\ldots)$ is a compatible coefficient family, with
a remainder of order $|t-s|$.
The same estimates, together with dominated convergence, show that
the $G_j$ are boundedness-preserving and strongly left-continuous.
Moreover, the map
\[
\Gamma_{G_0}(\eta)
=
\bigl(\eta,\mathcal M(\eta)\bigr)
\]
is admissible in the sense of Definition~\ref{def:admissible-causal}. Indeed,
$\mathcal M(\eta)$ is continuous whenever $\eta$ is continuous,
and strong left convergence of stopped paths is preserved by
\eqref{eq:volterra-memory}.

Now consider on
$\Lambda_T^{m+e}$ the non-anticipative functional
\[
F_0(t,z_t):=\sigma(t,z(t)).
\]
Its coefficient hierarchy is the ordinary differential hierarchy
of $\sigma$. Applying Theorem~\ref{thm:functional-chain-rule} to
$F$ and $G$ yields a compatible hierarchy for
$
F_0\circ\Gamma_{G_0}
=
\Sigma.$
Since the memory component of $G_1$ vanishes,
the first-order chain rule gives
\[
\nabla_\omega\Sigma(t,\eta_t)[v]
=
D_y\sigma
\bigl(t,\eta(t),\mathcal M_t(\eta)\bigr)[v],
\]
which is \eqref{eq:volterra-vertical-derivative}.
In particular, the results of Sections~\ref{sec:converse} and \ref{sec.chainrule} identify the first Gubinelli
coefficient intrinsically with the vertical derivative of the
 functional $\Sigma$.

It remains to verify that the resulting coefficient belongs to the
controlled-path class required for the fixed-point argument.
We do this locally, which is also the form in which the proof of
the RFDE theorem in
\cite{kwossek2025}
is carried out.

Fix $a<T$ and a continuous history
$\eta:[0,a]\to\mathbb R^m$.
For an extension $Y$ of $\eta$ to an interval
$I=[a,a+\delta]$, write
\begin{equation}
\label{eq:localized-memory}
\mathcal M_t^\eta(Y)
=
H_a^\eta(t)
+
\int_a^tK(t,r)\psi(Y_r)\,dr,
\qquad
H_a^\eta(t)
:=
\int_0^aK(t,r)\psi(\eta_r)\,dr.
\end{equation}
By boundedness of $\partial_1K$ and $\psi$,
\[
\|H_a^\eta\|_{1\text{-var};I}
\le C\delta.
\]
Moreover, differentiation of
\eqref{eq:localized-memory} gives
\[
\frac{d}{dt}\mathcal M_t^\eta(Y)
=
\partial_tH_a^\eta(t)
+
K(t,t)\psi(Y_t)
+
\int_a^t
\partial_1K(t,r)\psi(Y_r)\,dr,
\]
and therefore
\begin{equation}
\label{eq:memory-local-variation}
\|\mathcal M^\eta(Y)\|_{1\text{-var};I}
\le C\delta.
\end{equation}
If $Y,\widetilde Y$ are two extensions of the same past
$\eta$, then
\begin{align}
\|
\mathcal M^\eta(Y)-\mathcal M^\eta(\widetilde Y)
\|_{1\text{-var};I}
&\le
C\delta\,
\|Y-\widetilde Y\|_{\infty;I}
\nonumber\\
&\le
C\delta
\left(
|Y_a-\widetilde Y_a|
+
\|Y-\widetilde Y\|_{q\text{-var};I}
\right).
\label{eq:memory-local-lipschitz}
\end{align}
If $(Y,Y')$ is controlled by $\mathbf X$ on $I$, then
$
Z:=(Y,\mathcal M^\eta(Y)),
Z':=(Y',0),$
is again controlled by $\mathbf X$. Indeed, the first component
has the given controlled expansion, while the second component
has finite variation by \eqref{eq:memory-local-variation}.
Furthermore, on bounded subsets of the controlled-path space,
\eqref{eq:memory-local-lipschitz} yields
\begin{equation}
\label{eq:augmented-controlled-lipschitz}
\|
(Z,Z');
(\widetilde Z,\widetilde Z')
\|_{\mathbf X,q;I}
\le
C
\left(
|Y_a-\widetilde Y_a|
+
\|
(Y,Y');
(\widetilde Y,\widetilde Y')
\|_{\mathbf X,q;I}
\right).
\end{equation}

Applying the standard smooth-composition estimate for controlled
paths to
$
(t,Z_t)
\mapsto
\sigma(t,Z_t)
$
and using \eqref{eq:augmented-controlled-lipschitz}, we obtain,
on every controlled-path ball of radius $R$,
\begin{equation}
\label{eq:sigma-controlled-growth}
\|
\Sigma(Y),\Sigma'(Y,Y')
\|_{\mathbf X,q;I}
\le
C_R
\left(
1+\|(Y,Y')\|_{\mathbf X,q;I}
\right)^2
\left(
1+\|\mathbf X\|_{q;I}
\right)^2,
\end{equation}
and
\begin{align}
&
\|
\Sigma(Y),\Sigma'(Y,Y');
\Sigma(\widetilde Y),
\Sigma'(\widetilde Y,\widetilde Y')
\|_{\mathbf X,q;I}
\le
C_R
\left(
|Y_a-\widetilde Y_a|
+
\|
(Y,Y');
(\widetilde Y,\widetilde Y')
\|_{\mathbf X,q;I}
\right).
\label{eq:sigma-controlled-lipschitz}
\end{align}
The constants depend only on bounded sets of
$\|K\|_\infty$, $\|\partial_1K\|_\infty$,
$\|\psi\|_{C_b^2}$, $\|\sigma\|_{C_b^3}$,
the controlled-path norm and the rough-path norm.

Estimates
\eqref{eq:sigma-controlled-growth} and
\eqref{eq:sigma-controlled-lipschitz} show that, once the past
$\eta$ up to time $a$ is fixed, the coefficient 
\[
(Y,Y')
\longmapsto
\bigl(\Sigma(Y),\Sigma'(Y,Y')\bigr)
\]
is locally Lipschitz on controlled-path balls over
$I=[a,a+\delta]$. The same holds for the drift coefficient.
Together with the standard local estimates for rough integration,
this implies that, for $\delta$ and the rough-path variation on
$I$ sufficiently small, the Picard map associated with
\eqref{eq:volterra-rde} is a contraction on a controlled-path
ball.

Starting with $a=0$ gives a unique local solution. Once the
solution has been constructed on $[0,a]$, its restriction to this
interval fixes the history $\eta$, so the same argument may be
restarted on $[a,a+\delta]$. Since the coefficients are bounded
and the required local bounds are uniform on bounded
controlled-path sets, finitely many such intervals cover
$[0,T]$. The local solutions therefore concatenate to a unique
global solution.

Finally, for a solution of \eqref{eq:volterra-rde}, the standard
rough-integral expansion gives
\[
Y'_t
=
\sigma\bigl(t,Y_t,\mathcal M_t(Y)\bigr),
\]
while the controlled derivative of the integrand is obtained from
\eqref{eq:volterra-vertical-derivative},
\[
\Sigma'(Y,Y')_t
=
\nabla_\omega\Sigma(t,Y_t)[Y'_t]
=
D_y\sigma
\bigl(t,Y_t,\mathcal M_t(Y)\bigr)Y'_t.
\]
This proves the result.
\end{proof}

\begin{remark}
\label{rem:volterra-memory}
The 
local Lipschitz estimates used here are weaker than the global coefficient
assumptions used in \cite{kwossek2025}, but 
sufficient for the localized fixed-point argument used in the
proof.
The   contribution of Sections~\ref{sec:converse} and \ref{sec.chainrule} is the intrinsic identification of the controlled coefficients:
\[
\Sigma_j=\nabla_\omega^j\Sigma_0,\qquad j=1,2.
\]
\end{remark}
\section{Discussion}
\label{sec:discussion}

We have identified two complementary mechanisms
relating non-anticipative functional calculus \cite{ananova2017,cont2012,cont2019} and controlled rough paths. In the
forward direction, Proposition~\ref{prop:functional-taylor-expansion} shows that sufficiently regular
non-anticipative functionals generate compatible controlled Taylor
hierarchies through their vertical derivatives. Under horizontal
Lipschitz regularity alone, the resulting remainder exponents are
level-dependent and need not satisfy the classical linear grading of
controlled rough paths. Theorem~\ref{thm:weak-horizontal-rough-integral} identifies when these weaker
estimates nevertheless suffice for rough integration. Proposition~\ref{prop:classical-remainder} further shows that the
loss produced by the direct approximation argument is not always
intrinsic: for
\[
\sqrt2-1<\gamma\le\frac12
\]
a compensated sewing argument recovers the classical
$2\gamma$ controlled remainder.

The converse direction is structurally different and constitutes our
main representation result: Theorem~\ref{thm:representation} shows that a
family of non-anticipative coefficient functionals which satisfies
compatible controlled Taylor estimates as the underlying control
varies cannot have arbitrary coefficients. Under strong left
continuity and the $o(r)$ remainder condition,
\[
G_j=\nabla_\omega^jG_0,\qquad j=1,\ldots,p,
\]
over continuous stopped paths. Thus compatibility
across controls turns the relative differential structure of
controlled rough paths into an intrinsic differential structure on
path space. In particular, the coefficient hierarchy is symmetric
and uniquely determined by its zeroth-order component. This should be
contrasted with the fixed-control setting, where Gubinelli
derivatives need not be unique without additional roughness or
non-degeneracy assumptions on the driving signal.

Section~\ref{sec.chainrule} shows that this representation is compatible with
nonlinear functional transformations. Theorem~\ref{thm:functional-chain-rule} proves that
admissible composition preserves the class of compatible coefficient
families: if two such families are composed, the coefficients
obtained from the higher-order chain rule again satisfy compatible
controlled Taylor estimates. Applying Theorem~\ref{thm:representation} to the transformed
family then identifies these coefficients with the successive
vertical derivatives of the transformed zeroth-order functional.
Thus compatible families form a class on which non-anticipative
functional composition admits an intrinsic differential calculus.

The Volterra equation studied in Section~\ref{sec:application-rde} illustrates a
consequence of this closure property. For
\[
\mathcal M_t(Y)
=
\int_0^tK(t,r)\psi(Y_r)\,dr,
\]
the coefficient
\[
\Sigma(t,Y_t)
=
\sigma\bigl(t,Y(t),\mathcal M_t(Y)\bigr)
\]
exhibits path-dependence when the kernel $K$ depends
nontrivially on its first variable: the evolution of the memory term
contains
\[
\int_0^t\partial_1K(t,r)\psi(Y_r)\,dr
\]
and thus retains information about the  past trajectory.
At the same time, a vertical perturbation at the current time does
not alter the Lebesgue integral defining $\mathcal M_t$. Consequently,
\[
\nabla_\omega\Sigma(t,Y_t)
=
D_y\sigma\bigl(t,Y(t),\mathcal M_t(Y)\bigr),
\]
and the higher vertical coefficients are obtained from the same
chain-rule mechanism. The quantitative controlled-path estimates
needed for the associated RDE are separate from this structural
identification. For the Volterra class considered in Section~6 they
follow from the bounded regularity of the kernel and coefficients,
allowing the localized controlled-path fixed-point argument used in
rough functional differential equations to be applied with the past
history fixed. This illustrates the respective roles of the two
parts of the theory: Sections~\ref{sec:converse}--\ref{sec.chainrule} identify the intrinsic coefficient
hierarchy, while the quantitative controlled-path estimates provide
well-posedness of the resulting RDE.

A different approach has recently been developed by
Cuchiero, Guo and Primavera~\cite{cuchiero2025}, who consider
functionals defined directly on c\`adl\`ag weakly geometric rough
paths, i.e. which may  depend on the rough-path lift rather than
only on the underlying path. They introduce vertical
derivatives adapted to the Lie-group structure of rough-path space
and derive functional It\^o formulas and functional Taylor
expansions involving the full signature. An important distinction
is that their iterated  derivatives {\it need  not commute}. 
The vertical derivatives considered here are  derivatives along the vertical fibres of stopped
path space and are therefore symmetric \cite[Ch. 5]{cont2012}. Accordingly, the Taylor
polynomials of Proposition~\ref{prop:functional-taylor-expansion} involve  symmetric tensors
$X_{s,t}^{\otimes j}$; when paired with a geometric rough-path lift,
these correspond to contractions against the symmetric part of the
$j$th signature level.  If one is interested in functionals of the underlying path, Proposition \ref{prop:functional-taylor-expansion} provides functional expansions under fewer assumptions than
Cuchiero et al. \cite{cuchiero2025}, who need further regularity to control the asymmetric terms.

This suggests several directions for future investigation.  One is to seek a
representation theorem directly on rough-path space, replacing the
symmetric vertical hierarchy by a non-commutative system of
Lie-type derivatives capable of detecting the full signature.
Another is to develop quantitative versions of the stability theorem
of Section~\ref{sec.chainrule} which propagate power-type remainder estimates, rather
than only the abstract $o(r)$ compatibility condition. Such results
would connect the representation and chain-rule theory more directly
with rough integration and with well-posedness results for broader
classes of path-dependent rough differential equations.

\end{document}